\documentclass[11pt]{amsart}
\usepackage{geometry}                
\usepackage{graphicx}
\usepackage{amssymb}
\usepackage{epstopdf}
\title{Generalizations of the Kolmogorov-Barzdin embedding estimates}
\author{Misha Gromov and Larry Guth}

\newtheorem{lemma}{Lemma}[section]

\newtheorem{theorem}{Theorem}

\newtheorem{prop}[lemma]{Proposition}

\newtheorem{introtheorem}{Theorem}

\begin{document}

\begin{abstract} We consider several ways to measure the `geometric complexity' of an embedding from a simplicial complex into
Euclidean space.  One of these is a version of `thickness', based on a paper of Kolmogorov and Barzdin.  We prove inequalities
relating the thickness and the number of simplices in the simplicial complex, generalizing an estimate that Kolmogorov and Barzdin
proved for graphs.  We also consider the distortion of knots.  We give an alternate proof of a theorem of
Pardon that there are isotopy classes of knots requiring arbitrarily large distortion.  This proof is based on the expander-like properties
of arithmetic hyperbolic manifolds.
\end{abstract}

\maketitle

In this paper we study quantitative geometric estimates about embedding different spaces
into Euclidean space.  The main theme is the connection between topology and
geometry: if an embedding is topologically complicated, what kind of geometric
estimates does that imply?  Our results generalize a theorem by Kolmogorov and Barzdin \cite{KB}
from the 1960's about embedding graphs into $\mathbb{R}^3$.  Let's begin
by recalling what they did.

Given a topological embedding from a graph $\Gamma$ into $\mathbb{R}^3$, we say that
the embedding has thickness at least $T$ if the distance between any non-adjacent
edges is at least $T$, the distance between two vertices is at least $T$, and the
distance from an edge to a vertex not in the edge is at least $T$.  Roughly speaking,
one should imagine the vertices as balls of radius $T$ and the edges as (curved) tubes of
thickness $T$.  Kolmogorov and Barzdin mention as examples ``logical networks
and neuron networks" (\cite{KB} page 194).  A logical network probably refers to a computer
circuit where the edges correspond to wires and the vertices correspond to gates.  A neuron network
refers to a brain, where the vertices correspond to neurons and the edges correspond to
axons connecting the neurons.

Kolmogorov and Barzdin essentially proved the following theorem.

\begin{introtheorem} If $\Gamma$ is a graph of degree at most $d$ with $N$ vertices, then $\Gamma$ may be embedded 
with thickness 1 into a 3-dimensional Euclidean ball of radius $R \le C(d) N^{1/2}$.

 On the other hand, let $\Gamma$ be a random bipartite graph of degree $6$ with $2N$ vertices.  With high probability 
(tending to 1), there is no embedding of thickness 1 from $\Gamma$ into a ball of radius
$c N^{1/2}$. Moreover, if we embed $\Gamma$ into $\mathbb{R}^3$ with thickness 1, then 
the volume of the 1-neighborhood of the image is at least $c N^{3/2}$.
\end{introtheorem}

This result and its proof contain several interesting geometric ideas.  The most important idea is the discovery
of expanders.  
Kolmogorov and Barzdin essentially observed that a random graph is an expander.  Then they proved
that expanders are hard to embed in Euclidean space.

In this paper, we generalize the work of Kolmogorov and Barzdin in various directions.  
We will study objects with some properties analogous to expanders, and we will
see that they are hard to embed in Euclidean space.  We have three main results.  
The first is a higher-dimensional version of Theorem 1, dealing with embeddings of k-dimensional
simplicial complexes into n-dimensional space for $n \ge 2k+1$.  
The next two results have to do with the
closed arithmetic hyperbolic manifolds - a class of manifolds with some expander-like properties.  The
second theorem deals with the difficulty of embedding an arithmetic hyperbolic manifold in Euclidean space.

The third result, which is probably the most interesting, deals with certain complicated knots. 
 We recall  that a knot $K$ has distortion at least $D$ if there are two points $x,y \in D$ with $dist_K (x,y) 
\ge D dist_{\mathbb{R}^3}(x,y)$.  (Here $dist_K(x,y)$ denotes the distance from $x$ to $y$ along
$K$ - the length of the shortest segment of $K$ from $x$ to $y$.)  
For a long time, it was an open problem whether there
are isotopy classes of knots requiring arbitrarily large distortion.  Recently in \cite{P}, Pardon gave
a lower bound for the distortion of torus knots.  In particular, he gave a sequence of torus
knots that require arbitrarily large distortion.  We give a second
proof that there are isotopy classes of knots requiring arbitrarily large distortion.  Our proof is
based on the geometry/topology of arithmetic hyperbolic 3-manifolds.  A closed arithmetic hyperbolic
3-manifold $M$ admits a degree 3 cover of $S^3$ ramified only over a knot $K(M)$.  The `expander'-like properties
of $M$ imply `expander'-like properties of $K(M)$, which we use to show that $K(M)$ has a large distortion.

Now we describe our results in more detail.

First we generalize the theorems of Kolmogorov and Barzdin to higher dimensions.  
We consider topological embeddings $X^k \rightarrow \mathbb{R}^n$ where $X$ is a k-dimensional
simplicial complex.  We say that an embedding has combinatorial thickness
at least $T$ if the distance between the images of any two non-adjacent simplices is at least $T$.  
By a general position argument, $X$ embeds in $\mathbb{R}^n$
for all $n \ge 2k+1$.  We make this more quantitative by estimating the thickness of such an
embedding.  When $k=1$ and $n=3$, our result recovers the Kolmogorov-Barzdin theorem up to 
small errors.

\begin{introtheorem} Suppose that $n \ge 2k+1$ and that each vertex of $X$ belongs to at most $L$ k-faces of $X$.  Suppose
that $X$ has $N$ vertices. Then $X$ embeds with thickness 1 into $B^n(R)$ for $R = c(n, L, \epsilon) 
N^{\frac{1}{n-k} + \epsilon}$.

On the other hand, for each $n \ge 2k+1$ and each $\epsilon > 0$, we can find a sequence of k-complexes
$X_i$ obeying the following estimates.  The complex $X_i$ has $N_i \rightarrow \infty$ simplices.  
Each vertex in any $X_i$ lies in at most $L(\epsilon)$ simplices
of $X_i$.  Here $L(\epsilon)$ is a constant depending on $\epsilon$ but uniform among all of the $X_i$.  
Finally, if we embed $X_i$ into $\mathbb{R}^n$ with combinatorial thickness 
$\ge 1$, then the 1-neighborhood of the image has volume at least $c(n, \epsilon) N_i^{\frac{n}{n-k} - \epsilon}$.
In particular, if $X_i$ is embedded with combinatorial thickness $\ge 1$ into an n-ball of radius $R$, then
$R \ge c(n, \epsilon) N_i^{\frac{1}{n-k} - \epsilon}$.

\end{introtheorem}

In order to generalize Kolmogorov and Barzdin's examples, one might hope to find ``higher-dimensional
expanders".  We will give this a precise meaning in Section 2.  It's not hard to check that if $X^k$ is a
a higher-dimensional expander with $N$ simplices, and $X$ is embedded in an n-ball of radius $R$, then
$R \ge c N^{\frac{1}{n-k}}$.
However, we don't know whether such higher-dimensional expanders exist.  The examples we
consider are the k-skeleta of high-dimensional cubical lattices.  These are not as good as expanders, so 
our lower bound has an extra $\epsilon$ that doesn't appear in the Kolmogorov-Barzdin theorem.

Our next theorem involves a different notion of the `thickness' of an embedding.
We say that an embedding $I: X \rightarrow \mathbb{R}^n$ has retraction
thickness at least $T$ if the $T$-neighborhood of $I(X)$ retracts to $I(X)$.  When we first formulated
it, we thought this definition was just a minor variation of the Kolmogorov-Barzdin definition of combinatorial
thickness.  But it turns out that retraction thickness has very different properties.  For example, an
expander graph with $N$ vertices can be embedded with retraction thickness 1 into a 3-ball of radius
$\sim N^{1/3}$, much smaller than for combinatorial thickness.  The reason for this is roughly that the
retraction thickness usually treats homotopic spaces the same, and an expander graph with $N$
vertices is homotopic to a planar graph with $N$ vertices.

We prove several results of the following flavor.  If $X$ is homotopically complicated and we embed
$X$ in $\mathbb{R}^n$ with retraction thickness at least 1, then the volume of the 1-neighborhood of
$I(X)$ should be big.  For example, we will prove that
the volume of the 1-neighborhood is $\gtrsim$ the sum of Betti numbers of $X$.  Our most interesting result in this direction involves arithmetic hyperbolic manifolds.

\begin{introtheorem} Let $X$ be a closed arithmetic hyperbolic k-manifold with volume $V_{hyp}$.  Suppose
$k \ge 3$.  If we embed
$X$ in $\mathbb{R}^n$ with retraction thickness 1, then the 1-neighborhood of the image has
volume at least $c(k,n) V_{hyp}^{\frac{n}{n-1}}$.

\end{introtheorem}

This theorem also involves expanders in a sense.  The key property of arithmetic hyperbolic manifolds
is that they obey an expander-type isoperimetric inequality.  In particular, they come with natural
triangulations and their 1-skeleta are expander graphs.  This expander property is exploited in the proof.
A non-arithmetic closed hyperbolic 3-manifold with volume $V_{hyp}$ obeys a weaker estimate: the 1-neighborhood
of the image has volume $\gtrsim V_{hyp}$.

It's not clear how sharp this theorem is.  For example, suppose that we embed our arithmetic
hyperbolic 3-manifold $X^3$ into $B^7(R)$ with retraction thickness at least $1$.  Our theorem shows that 
$R \gtrsim V_{hyp}^{1/6}$.  On the other hand, we will construct such embeddings with
$R \sim V_{hyp}^{1/4 + \epsilon}$.  We don't know where the optimal value of $R$ lies
within this range.

Our last result concerns the geometric properties of complicated knots.  A knot is an embedding
from $X = S^1$ into $\mathbb{R}^3$.  In this case, the space $X$ is simple, but we consider complicated
isotopy classes of embeddings.  The knots we study are built using arithmetic hyperbolic
3-manifolds.
To define these knots, we recall the following theorem of Hilden and Montesinos (\cite{H} and \cite{M}).

\newtheorem*{hmthm}{Ramified cover theorem}

\begin{hmthm} (Hilden, Montesinos) Any closed oriented 3-manifold $M$ admits a degree 3 map
$M \rightarrow S^3$ which is a ramified cover ramified over a knot $K = K(M)$.
\end{hmthm}

Because $M$ is very complicated topologically, it turns out that $K(M)$ must be complicated geometrically.
In particular, we prove that $K(M)$ has a large distortion.

\begin{introtheorem} Let $M$ be a closed arithmetic hyperbolic 3-manifold with volume $V$.
Suppose that $M$ is a 3-fold cover of $S^3$ ramified over a knot $K(M)$.  Then $K(M)$ has distortion at least $c V$.
(And so does any knot isotopic to $K(M)$.)

\end{introtheorem}

\noindent In particular, we see that there are isotopy classes of knots requiring arbitrarily large distortion.

There are many different ways to measure the ``geometric complexity" of an embedding.  We consider
several different measures of complexity, related to combinatorial thickness, retraction thickness, and
distortion.  Each of these ways of measuring geometric complexity leads to different kinds of estimates.
There are many open questions along these lines, and some of them are indicated in the paper at
the end of each section.

In the first section of the paper, we discuss the work of Kolmogorov and Barzdin.
In the second section of the paper, we discuss embeddings of higher dimensional complexes
and prove Theorem 2.  In the third section, we discuss embeddings of homotopically
complicated spaces and prove Theorem 3.  In the fourth section, we discuss the distortion of knots
and prove Theorem 4.  The different sections are mostly independent.  In an appendix
we review some fundamental facts about the geometry and topology of arithmetic hyperbolic manifolds.  
In particular, we emphasize the expander-type properties of these manifolds and the connections
between expanders and topology.

Acknowledgements.  Over the course of working on the paper, the second author was supported by NSERC, by NSF grant DMS-0635607, and by the Monell
foundation.  

\section{Embedding networks in Euclidean space}

In this section, we discuss the paper ``On the realization of networks in three-dimensional space"
by Kolmogorov and Barzdin.  In Section 1.1, we summarize the main ideas of the paper.  We give
detailed proofs of any results that we need later on.  In Section 1.2, we make some historical comments.

\subsection{Embedding graphs in $\mathbb{R}^3$}

Suppose that $\Gamma$ is a graph, and $I: \Gamma \rightarrow \mathbb{R}^3$ is a topological embedding of $\Gamma$
in three-dimensional Euclidean space.  We think of $\Gamma$ as a simplicial complex of dimension 1.
We say that the embedding $I$ has combinatorial thickness at least
$T$ if $dist[ I(\Delta_1), I(\Delta_2) ] \ge T$ whenever $\Delta_1$ and $\Delta_2$ are non-adjacent simplices
of $\Gamma$.  (Here $\Delta_i$ is either a vertex of $\Gamma$ or an edge of $\Gamma$.  So we see that
the distance between any two distinct vertices is at least $T$, and the distance between any two edges that
don't share a vertex is at least $T$, and the distance between an edge and a vertex not in the edge is
at least $T$.)

For a first perspective, suppose that the graph $\Gamma$ is an $N^{1/2} \times N^{1/2}$ grid.
It is easy to embed $\Gamma$ with combinatorial thickness 2 into a block of dimensions $10 \times 10 N^{1/2}
\times 10 N^{1/2}$.  Moreover, we can assume that the distance from $\Gamma$ to the edge of the block is
at least 2.  This block can in turn be folded up to fit in a ball.  More precisely, there is an embedding $\Psi$
from the block into a ball of radius $100 N^{1/3}$ which is locally 2-bilipschitz.  This means that $\Psi$
stretches vectors by at most a factor of 2 and contracts vectors by at most a factor of 2.  In particular, the image
$\Psi(\Gamma)$ has combinatorial thickness at least 1.  So we see that a grid admits an embedding of thickness
1 into a ball of radius only $\sim N^{1/3}$.  But it turns out that other graphs of bounded degree are much harder
to embed than grids.  These other graphs are expanders, and we discuss them more below.

In \cite{KB}, Kolmogorov and Barzdin proved that random graphs are the hardest graphs to embed and gave a sharp
estimate for the thickness of embedded graphs.  They essentially proved Theorem 1 from the introduction.  For 
convenience, we copy the statement here.

\begin{theorem} If $\Gamma$ is a graph of degree at most $d$ with $N$ vertices, then $\Gamma$ may be embedded 
with thickness 1 into a 3-dimensional Euclidean ball of radius $R \le C(d) N^{1/2}$.

On the other hand, let $\Gamma$ be a random bipartite graph of degree $6$ with $N$ vertices.  With high probability 
(tending to 1), there is no embedding of thickness 1 from $\Gamma$ into a ball of radius
$c N^{1/2}$. Moreover, if we embed $\Gamma$ into $\mathbb{R}^3$ with thickness 1, then 
the volume of the 1-neighborhood of the image is at least $c N^{3/2}$.
\end{theorem}

In the first part of their paper, Kolmogorov and Barzdin construct thick embeddings from $\Gamma$ to
$B(R)$, $R= C(d) N^{1/2}$.  The vertices of their graph are mapped to the
boundary of the ball.  According to Arnold's reminiscences of Kolmogorov, the original motivation for the
research was the study of the brain \cite{A}.  Kolmogorov had heard that all neurons are located at the edge
of the brain and that the large center of the brain consists of axons between the neurons.  The embedding
they construct has the same architecture: the vertices lie on the edge of the ball, and the edges fill up
the inside of the ball.  After placing the vertices evenly around the boundary of the ball, in an arbitrary order,
they construct the edges of the graph one at a time.  Each edge is a piecewise linear curve depending on 
a couple of parameters.  By a counting argument, they show that at each step they can choose these parameters
to avoid all of the previous edges.

They prove that this construction is sharp for random graphs.  To do this, they prove that random
graphs are expanders with high probability.  Recall that a graph $\Gamma$ is called an $\alpha$-expander
if it obeys the following isoperimetric inequality: if $A$ is a subset of the vertices of $\Gamma$ with at most
half of the vertices, then the number of edges from $A$ to the complement of $A$ is at least $\alpha |A|$.
There exists a constant $\alpha > 0$ so that a random bipartite graph of degree 6 
with $N$ vertices is an $\alpha$-expander with probability tending to 1 as $N$ goes to infinity.  (Kolmogorov
and Barzdin proved something very similar to this - see Section 1.2 for more information.  The expander properties
of random graphs are discussed in detail in \cite{HLW}.)

Suppose that $\Gamma$ is an $\alpha$-expander with degree 6 and $N$ vertices, embedded with thickness 1 
into a ball of radius $R$.  Slice the ball into
horizontal planes $z = h$.  By choosing $h$, we can arrange that at least half of the points lie in the region $\{ z \ge h\}$
and at least half lie in $\{ z \le h \}$.  Let $V^+$ denote the vertices in $\{ z \ge h \}$ and $V^-$ the vertices
in $\{ z \le h \}$.  By the expander property, there are $\ge (1/2) \alpha N$ edges of the graph from $V^+$ to $V^-$, and each
of these edge must intersect the plane $z = h$.  Any edge shares a vertex with $\le 2 d$ other edges, and so
we can choose $\gtrsim N$ edges of $\Gamma$ passing through $\{ z = h\}$ with no two edges sharing
a vertex.  By definition of combinatorial thickness, the unit neighborhoods around these edges are
disjoint.  Let $D$ denote the disk given by intersecting the ball $B^3(R)$ with the plane $\{ z = h \}$.  It
has radius at most $R$, yet it contains $\gtrsim N$ points pairwise separated by a distance of at least 1.
So we see that $N \lesssim R^2$, and so the radius $R$ is $\gtrsim N^{1/2}$.  Therefore, the radius
in the construction of Kolmogorov-Barzdin cannot be improved by more than a constant factor.

Next we turn to the volume of the 1-neighborhorhood of $I(\Gamma)$.
We write $U := N_1( I(\Gamma) )$ to denote the radius 1
neighborhood of the image $I(\Gamma)$.  If the embedding $I$ has thickness 1, we need to prove
that the volume of $U$ is $\gtrsim N^{3/2}$.  To do this, we will again find a plane that bisects 
the vertices of $\Gamma$.   We find the plane using the Falconer slicing
inequality \cite{F}.

\newtheorem*{falc}{Falconer Slicing Estimate}

\begin{falc} Let $U$ be an open set in $\mathbb{R}^n$.  If $k > n/2$, then we can slice $U$ with parallel $k$-planes so that the intersection of $U$ with each plane has $k$-area at most $ C_n Vol[U]^{k/n} . $
\end{falc}

In our case, $n=3$ and $k = 2 > n/2$, and so we can slice $U$ into parallel planes so that each plane intersects $U$ in area $\lesssim Vol(U)^{2/3}$.
One of the planes bisects the vertices of $\Gamma$, and we can conclude by the argument above that
$N \lesssim Vol(U)^{2/3}$.  

For context, we briefly discuss the values of $k$ and $n$ in Falconer's inequality.
The restriction on $k$ in Falconer's estimate is not well-understood.  If $k=1$, the theorem is definitely false because of a construction of Besicovitch.  
For $2 \le k \le n/2$, I believe it's an open problem whether the theorem still holds.  In all of our applications, we will have $k > n/2$.
Nevertheless, for smaller values of $k$, $U$ may be swept out by curved k-dimensional surfaces of small area as shown in \cite{Gu}.

\subsection{Historical comments}

Expander graphs are an important object in mathematics.  At the first or second encounter, it seems
very surprising that they exist.  They have applications in many areas, including computer science,
geometry, and number theory.  The essay \cite{HLW} gives a thorough expository account of this area.  Expander
graphs also have some applications in topology which we discuss more in the appendix.

The first proof that expanders exist is usually cited as Pinsker's proof, published in 1973 \cite{Pi}.
But the paper \cite{KB} of Kolmogorov and Barzdin also essentially proves that a random graph
is an expander.  It was published in 1967.

Kolmogorov began to work on the graph embedding project in the early 1960's.  He did the graph
embedding construction described above.  Then he gave an example where an embedding of thickness
1 required a ball of volume $N^{3/2} / log N$.  We don't know what example Kolmogorov was using.  
Later, Barzdin joined the project and proved the sharper estimate $N^{3/2}$.
  
Strictly speaking, Kolmogorov and Barzdin studied directed graphs where the number of incoming edges at each
vertex was always $d$ (say $d=2$), but the number of outgoing edges could vary.  If we consider these
graphs as undirected graphs, the degree is not actually bounded.  Random directed graphs are easy to define.  Each
vertex has $d$ incoming edges, and each of these edges is assigned a starting vertex uniformly at random.  (It looks
plausible that Kolmogorov and Barzdin work with directed graphs because it is easier to define a random
directed graph.)  A random directed graph with $d=2$ and $N$ vertices 
typically has a few vertices with
something like $\log N$ outgoing edges.  If we look at them as undirected graphs, they typically don't have bounded degree.
Still, the results of Kolmogorov and Barzdim come very close to producing expander graphs of uniformly bounded
degree.  For example, their graphs do have bounded average degree.  Probably one can build bounded-degree expanders
by a short additional argument that prunes the high-degree vertices and a few other vertices from these graphs.

Barzdin's recollections of the paper are interesting.  These appear in the notes section of the collected works of
Kolmogorov \cite{B}.  Kolmogorov credits Barzdin with the ``very idea of setting the question of almost all
graphs".  It sounds like Barzdin had the idea of looking at random graphs and proving that most of them are expanders.
Barzdin writes modestly, ``In fact I only gave new proofs (and somewhat generalized) theorems obtained by Andrei Nikolayevich
earlier, so that my achievements here are not very important."  

On the origins of the paper, Barzdin writes, ``Unfortunately, I don't remember what was the occasion or event at which 
Andrei Nikolayevich first mentioned these results (I was not present there).  I know only that the topic discussed
there was the explanation of the fact that the brain (for example that of a human being) is so constituted that 
most of its mass is occupied by nerve fibers (axons) while the neurons are only disposed on its surface.  The
construction of Theorem 1 precisely confirms the optimality (in the sense of volume) of such a disposition of the
neuron network."

The paper \cite{KB} has possible connections with other parts of science.  Kolmogorov and Barzdin write, ``Examples
of such networks are logical networks and neuron networks.  It is precisely for this reason that the question of 
constructing such networks in ordinary three-dimensional space under the condition that the vertices are balls while
the edges are tubes of a certain positive diameter is of importance."  We don't know whether connections with
neuroscience or computer science have been pursued.  On the neuroscience side, to what extent is it true
that the neurons of the brain lie on the boundary of a ball with the axons going through the interior of the ball?  To what extent is
the `graph' of the brain an expander?  Another possible connection is with the design of computer chips.  In his essay \cite{A},
Arnold writes, ``It is interesting to note that this paper \cite{KB} has remained little known even to specialists, perhaps
because the mathematical exposition is too serious.  When I mentioned it in a paper in {\it Physics Today} dedicated
to Kolmogorov (October 1989), I received a sudden deluge of letters
from American engineers who were apparently working in miniaturization of computers, with requests for a precise
reference to his work."  How do engineers design
the 3-dimensional structure of computer chips?  Are the underlying graphs expanders?

There is a substantial literature on the geometry of computer chips, which I know little about.  Paule Beame pointed
me to the paper \cite{BK} by Brent and Kung 
which studies the geometry of a computer chip which multiplies two n-digit binary numbers.  Brent and Kung give sharp
estimates involving the area and running time of the computer chip.  

(The computer circuits that Brent and Kung study are laid
out in planar domains with $\lesssim 1$ wires crossing at any point in the planar domain.  
Nearly planar arrangements of wires are not good for embedding expanders in space.
If an expander graph $\Gamma$ is embedded with thickness 1 into a rectangular box with
dimensions $ A \times B \times C$, $ A \le B \le C$, then the number of vertices of $\Gamma$ is 
$\lesssim AB$ - it is bounded by the smallest cross-sectional area of the box.  In particular, if $\Gamma$
has $N$ vertices, then it may embed into a box of dimensions $10 \times L \times L$ only if $L \gtrsim N$,
and only if the box has volume $\gtrsim N^2$.  Brent and Kung explain the nearly planar arrangement of
computer circuits: ``Because of heat-dissipation, packing, and testing requirements, a two-dimensional planar
model is reasonable." )

\section{Embeddings of higher-dimensional complexes}

In this section, we generalize the upper and lower bounds of Kolmogorov-Barzdin to embeddings from k-dimensional
simplicial complexes into $\mathbb{R}^n$ for $ n \ge 2k + 1$.

Suppose that $X$ is a simplicial complex and $I: X \rightarrow \mathbb{R}^n$ is an embedding.  We say
that $I$ has combinatorial thickness $\ge T$ if the distance between $I(\Delta_1)$ and $I(\Delta_2)$
is at least $T$ for any two non-adjacent simplices of $X$.  These simplices may have any dimension - they
don't need to be top-dimensional.  Two simplices are non-adjacent if they share no vertices, or equivalently
if their closures are disjoint.

We prove if $n \ge 2k+1$, then any simplicial complex $X^k$ with $N$ simplices and bounded local geometry can be embedded
with combinatorial thickness 1 into an n-ball of radius $\lesssim N^{\frac{1}{n-k} + \epsilon}$, and we prove that this estimate
is sharp up to factors of $N^\epsilon$.  In Section 2.1, we construct the embeddings.  In
Section 2.2, we give examples when the estimate is nearly sharp.  In Section 2.3, we discuss
open problems.

\subsection{Constructing embeddings}

Suppose that $X$ is a simplicial complex of dimension $k$.  For $n \ge 2k+1$, the complex embeds in $\mathbb{R}^{n}$ by 
general position arguments.  We will estimate the radius of a ball $B^n(R)$ so that $X$ embeds with combinatorial
thickness 1 into $B^n(R)$.  Our proof is a quantitative version of the standard general position argument.  
But there is one subtlety: we need to do a quantitative general position argument at
two different scales.

For $k=1$, $n=3$, Kolmogorov and Barzdin proved that a graph with degree $\le d$ admits an embedding with thickness 1 into a ball of radius $R \sim C(d) N^{1/2}$.  Their method generalizes immediately to $k=1$ and any $n \ge 3$.  In higher dimensions, the radius goes like $R \sim C(d) N^{\frac{1}{n-1}}$.  This exponent is sharp for all $n \ge 3$ when $X$ is an expander.  

In higher dimensions, the right generalization appears to be the following:

\newtheorem*{conjecture}{Conjecture}
\begin{conjecture} Suppose that $X$ is a k-dimensional simplicial complex with $N$ simplices, and that each vertex lies in $\le L$
simplices.  Suppose that $n \ge 2k+1$.  Then there is
a 1-thick embedding from $X$ into the n-dimensional ball of radius

$$ R \le C(n,L)  N^{\frac{1}{n-k}} . $$

\end{conjecture}

When $k=1$, this conjecture agrees with the Kolmogorov-Barzdin embedding construction, and when $k=0$ it is easy to check.  In all dimensions, we prove this conjecture up to some logarithmic factors.

\newtheorem*{theorem2.1}{Theorem 2.1}

\begin{theorem2.1} Suppose that $X$ is a k-dimensional simplicial complex with $N$ vertices and each vertex lying in at most $L$ simplices.  Suppose that $n \ge 2k+1$.  Then
$X$ embeds with combinatorial thickness $1$ into the n-dimensional ball of radius

$$R \le C(n,L) N^{\frac{1}{n-k}} (\log N)^{2k+2}. $$

\end{theorem2.1}

We first consider random facewise linear embeddings.  A random facewise linear embedding means that we first map each vertex of $X$ randomly to a point in a given ball.  
Then we extend the map linearly on each simplex.  Random facewise linear embeddings have thickness much too
small to prove our theorem, but we
will construct our embedding by modifying a random facewise linear embedding.  To get a perspective, we first compute the thickness of a random facewise
linear embedding.

\begin{prop} Let $X$ be a k-dimensional simplicial complex with $N$ simplices.  Suppose that $n \ge 2k+1$.  A random face-wise linear 
embedding from $X$ into the n-dimensional ball of radius $R$ is usually $T$-thick for

$$ T \sim R N^{- \frac{2}{n-2k}} . $$

In particular, we can find a 1-thick embedding if

$$ R \gtrsim N^{\frac{2}{n-2k}} . $$

\end{prop}

\begin{proof} Fix $\Delta, \Delta'$ simplices in $X$ of dimensions $d, d' \le k$.  Suppose that their vertex sets are disjoint.  Then the  probability that the distance from $I(\Delta)$ to $I(\Delta')$ is less than $T$ is roughly $(T / R)^{n - d - d'} \le  (T / R)^{n-2k}$.  Now the number of pairs of simplices is at most $N^2$, and so we get a $T$-thick embedding most of the time provided that

$$C(n) N^2 (T / R)^{n - 2k } < 1/2. $$

\end{proof}

This random embedding is not nearly as good as the Kolmogorov-Barzdin embedding in case $k=1$.  We will improve it by bending the simplices at a smaller scale.  Roughly speaking, we will put our map in general position at
scale $N^{\frac{1}{n-k}}$ and also at scale 1.  Now we turn to the detailed construction of our embedding.

\proof We write $A \lesssim B$ for $A \le C(n, L) B$.

Let $I_0$ be a random facewise linear embedding from $X$ into $B^n(N^{\frac{1}{n-k}})$.  We will deform $I_0$ to an embedding with thickness $\gtrsim (\log N)^{-(2k+2)}$.  
To begin our analysis, we count how many simplices meet unit balls $B(p,1) \subset B^n$.

\begin{lemma} With high probability, each unit ball $B(p,1) \subset B(N^{\frac{1}{n-k}})$ meets $\lesssim \log N$ simplices of $I_0(X)$.
\end{lemma}

\proof Fix a ball $B(p,1)$ and a simplex $\Delta^d \subset X$.  The probability that $I_0(\Delta^d)$ meets $B(p,1)$ is 
$\lesssim N^{-1}$.  
We can see this as follows.  The probability is essentially independent of which ball $B(p,1)$ we are considering. So we can first
pick $I_0(\Delta^d)$ and then pick a random $p$ and look at the probability that $B(p,1)$ meets $I_0(\Delta^d)$.  We consider the projection
into the plane orthogonal to $I_0(\Delta^d)$.  If $B(p,1)$ meets $I_0(\Delta^d)$, then the projection of $p$ into this plane needs to
lie within a unit distance of the point corresponding to $I_0(\Delta^d)$.  Now the projection of $p$ is basically equally likely to lie anywhere in a ball of radius $R = N^{\frac{1}{n-k}}$
and dimension $n-d \ge n-k$.

Therefore, the expected number of simplices hitting $B(p,1)$ is $\lesssim 1$.  But some balls may meet more simplices than average.

Label the simplices of $X$ $\Delta_1, \Delta_2, ...$  Consider the event that $\Delta_i$ meets $B(p,1)$.  If these events were independent, then
the probability of having $A$ simplices meet $B(p,1)$ would decay $\le {N \choose A} (C/ N)^{A} \le C^A / (A!) \le e^{-A}$ for large $A$.  If that were true, then
with high probability, the number of simplices hitting $B(p,1)$ would be $\lesssim \log N$ for every $p$.

If a set of simplices has no common vertices, then the positions of the simplices really are independent.  So we can finish with a coloring
trick.  Color the simplices of $X$ so that any two simplices sharing a vertex are a different color.  Let $\beta$ denote the number of colors.
Because each vertex lies in $\lesssim L$ simplices, we can bound the number of colors by $(k+1) L$.  With high probability, the number of simplices
of each color meeting $B(p,1)$ is $\lesssim \log N$ for each $p$.  Since the number of colors is $\lesssim 1$, we are done. \endproof

The number of different simplices that intersect a unit ball $B(p,1) \subset B^n(N^{\frac{1}{n-k}})$ is $\lesssim \log N$.  
In some unlucky balls, these simplices pass extremely
close to each other.  We are going to bend the simplices at scale 1 to fix this problem.  By putting simplices in general position within each
unit ball, we will show that they don't need to be any closer to each other than $(\log N)^{-(2k+2)}$.

In order to bend the simplices, we need to decompose each simplex into finer simplices.

\newtheorem*{sdlemma}{Simplex decomposition lemma}

\begin{sdlemma}   We can refine the simplicial complex $I(X)$ into a finer complex $X'$ with the following properties.

1. Each simplex of $X'$ has diameter $\lesssim 1$.

2. Each vertex of $X'$ lies in $\lesssim L$ simplices of $X'$.

3. Each unit ball in $\mathbb{R}^n$ intersects $\lesssim \log N$ simplices of $X'$.

\end{sdlemma}

\begin{proof} Intersect the complex $X$ with a unit cubical lattice $K$.  This leads to a polyhedral structure on $X$ which
refines the original simplicial structure.  Each face of the new structure lies in a unit cube, so it has diameter $\lesssim 1$.
The number of faces of $\Delta \cap K$ which meet a unit ball is $\lesssim 1$.  Now the faces of this polyhedral structure
may not be simplices, but we subdivide them each into simplices.  The number of simplices needed to triangulate each
face is $\lesssim 1$, because there are only finitely many combinatorial possibilities for the faces.  If each vertex of $X$
lies in $\le L$ simplices, then each vertex of $X'$ lies in $\le CL$ simplices. 

Each unit ball in $\mathbb{R}^n$ intersects $\lesssim \log N$ simplices of $X$.  Let $\Delta$ be any simplex of $X$
that intersects $B(p,1)$.  The intersection $\Delta \cap B(p,1)$ is contained in $\lesssim 1$ simplices of $X'$.  \end{proof}

The complex $X'$ is homeomorphic to $X$, with each simplex of $X$ corresponding to many simplices of $X'$.

We now move each vertex of the refined simplicial complex $X'$ by a vector of length $\le 1$, and we extend linearly on the refined simplices.  In this
way, we get a map $I: X' \rightarrow B(N^{\frac{1}{n-k}} + 1)$.  We will check that for a good choice of these vectors, 
the combinatorial thickness of $I$ is $\gtrsim (\log N)^{-(2k+2)}$.  

There is a slight issue with language at this point.  The map $I$ is an embedding from $X'$ to 
$\mathbb{R}^n$ and we can also view it as an embedding from $X$ to $\mathbb{R}^n$.  The combinatorial
thickness depends on whether we use the simplicial structure of $X$ or the simplicial structure of $X'$.
Since $X'$ is a refinement of $X$, the combinatorial thickness with respect to $X'$ is at most the combinatorial
thickness with respect to $X$.  We state this as a lemma.

\begin{lemma} Suppose that $X'$ is a refinement of $X$.  Let $I: X' \rightarrow \mathbb{R}^n$ be an embedding.  Let $T'$ denote the combinatorial thickness of $I$ as an embedding from $X'$ to $\mathbb{R}^n$, and let $T$ denote the combinatorial thickness of $I$ as an embedding from $X$ to $\mathbb{R}^n$.  Then $T' \le T$.
\end{lemma}

\begin{proof}  Consider two non-adjacent simplices $\Delta_1, \Delta_2$ in $X$.  Let $d$ denote the
distance from $I(\Delta_1)$ to $I(\Delta_2)$.  We can find points $p_1$ in the closure of $\Delta_1$ and $p_2$ in
the closure of $\Delta_2$ so 
that $| I(p_1) - I(p_2) | = d$.  Let $\Delta_1' \subset \Delta_1$ be the smallest simplex of $X'$ containing $p_1$, and
similarly for $p_2$.  Since $X'$ is a refinement of $X$, the simplex $\Delta_1'$ is contained in the closure of $\Delta_1$
and $\Delta_2'$ is contained in the closure of $\Delta_2$.  Hence $\Delta_1'$ and $\Delta_2'$ are non-adjacent.
We conclude that $d$ is at least the distance between $I(\Delta_1')$ and $I(\Delta_2')$ which is at least $T'$.  So
$T \ge T'$. \end{proof}

Therefore, it suffices to choose $I$ so that the combinatorial thickness of $I: X' \rightarrow B(N^{\frac{1}{n-k}} + 1)$ 
is $\gtrsim (\log N)^{-(2k+2)}$. 

Let $\{ v_i \}$ denote the vertices of $X'$.  We will define $I(v_i)$ one vertex at a time, and then extend $I$ linearly to each simplex of $X'$.
We will choose $I(v_i)$ in the unit ball around $I_0(v_i)$.  Let $Con(i)$ denote this unit ball with probability measure given by 
renormalizing the volume measure. Then let $Con$ denote the product $\prod Con(i)$ with the product measure.  Here $Con$ 
is the configuration space of all of our choices, and we have to find an element of $Con$ so that the thickness of $I$ is pretty big.  (Remark: A random element of $Con$ does not usually have the desired properties, but introducing this
probability measure is still useful in our proof.)

For simplices $\Delta_1, \Delta_2 \subset X'$ with no common vertices, we define $Bad_\epsilon(\Delta_1, \Delta_2)$ to be the set of
all the configurations where the distance from $I(\Delta_1)$ to $I(\Delta_2)$ is $\le \epsilon$.  We will choose an $\epsilon \sim
(\log N)^{-(2k+2)}$, and it suffices to find a configuration which does not lie in any of the sets $Bad_\epsilon(\Delta_1, \Delta_2)$.

The main helpful point is that many of the sets $Bad_\epsilon(\Delta_1, \Delta_2)$ are empty.  If $dist(\Delta_1, \Delta_2) > 3$, then no matter how we move the vertices by distances $< 1$, $I(\Delta_1)$ and $I(\Delta_2)$ will still be separated by at least 1.  Since we will choose $\epsilon < 1$, $Bad_\epsilon(\Delta_1, \Delta_2)$ will be empty.  
From now on, we only discuss the non-empty sets $Bad_\epsilon(\Delta_1, \Delta_2)$.
By Lemma 2.2 and the simplex decomposition lemma, there are $\lesssim \log N$ simplices of $X'$ in any ball of radius 10.  In particular, for each vertex $v_i$, there are $\lesssim \log N$ non empty sets  $Bad_{\epsilon}(\Delta_1, \Delta_2)$ where $v_i$ lies in $\Delta_1$ or $\Delta_2$.

We can also estimate the probability of $Bad_\epsilon(\Delta_1, \Delta_2)$ in the non-empty case.  

\begin{lemma} For any simplices $\Delta_1, \Delta_2$ with no vertex in common, 
the probability of $Bad_\epsilon(\Delta_1, \Delta_2)$ is $\lesssim \epsilon$.
\end{lemma}

This lemma is similar to the analysis of a random facewise linear embedding.  There is also a formal proof later in
Proposition 3.12.

Now we choose $v_1$, then $v_2$, etc.  As we make each choice, we keep track of the conditional probability of the
sets $Bad_\epsilon(\Delta_1, \Delta_2)$, given the choices we have made so far.  This analysis is similar to
the Lovasz local lemma in \cite{EL}.  

When we pick $v_i$, the conditional
probability of $Bad_\epsilon(\Delta_1, \Delta_2)$ changes only if $v_i$ is a vertex of $\Delta_1$ or $\Delta_2$.  So the
choice of $v_i$ affects $\lesssim \log N$ conditional probabilities.  The odds of increasing any conditional probability
by a factor $A$ are $\le (1/A)$.  Therefore, we can choose $v_i$ so that the conditional probabilities of $Bad_\epsilon(\Delta_1, \Delta_2)$
increase by at most a factor $\lesssim \log N$, when $v_i$ is a vertex of $\Delta_1$ or $\Delta_2$.

Each pair $(\Delta_1, \Delta_2)$ has in total $\le 2k+2$ vertices.  Therefore, after we have chosen all $v_i$, the conditional
probability of $Bad_\epsilon(\Delta_1, \Delta_2)$ has increased by a factor $\lesssim (\log N)^{2k+2}$.  Since the probability
of $Bad_\epsilon(\Delta_1, \Delta_2)$ was $ \lesssim \epsilon$ before making any choices, after all our choices, the conditional
probability of $Bad_\epsilon(\Delta_1, \Delta_2)$ is $\lesssim \epsilon (\log N)^{2k+2}$.  Of course, once we have chosen all $I(v_i)$,
the conditional probability of $Bad_\epsilon(\Delta_1, \Delta_2)$ is either 1 or 0, depending on whether our choice lies
in $Bad_\epsilon(\Delta_1, \Delta_2)$ or not.  We pick $\epsilon$ sufficiently small so that all conditional probabilities
of $Bad_\epsilon(\Delta_1, \Delta_2)$ are $< 1$.  This $\epsilon$ is $\sim (\log N)^{-(2k+2)}$.  Now our choice
of the map $I$ does not lie
in any of the sets $Bad_\epsilon(\Delta_1, \Delta_2)$.  In other words, the embedding $I: X' \rightarrow B(N^{\frac{1}{n-k}})$ has combinatorial thickness $\ge \epsilon \gtrsim (\log N)^{-(2k+2)}$.  \endproof

\subsection{Complexes that are hard to embed in Euclidean space}

In this section, we prove that some complexes are hard to embed in Euclidean space.  Our
result generalizes the result of Kolmogorov and Barzdin that some graphs are hard
to embed in $\mathbb{R}^3$, but our estimate is a little less sharp than theirs.  

\newtheorem*{theorem2.2}{Theorem 2.2}

\begin{theorem2.2} Fix $k$ and $n$ with $n \ge 2k+1$.  For any $\epsilon > 0$, there
is a constant $L(\epsilon)$, and a sequence of k-dimensional simplicial complexes
$X_i$ with the following properties.

Each vertex in each $X_i$ belongs to $\le L(\epsilon)$ simplices.

The number of simplices in $X_i$ is $N_i \rightarrow \infty$.

If $X_i$ is embedded into $\mathbb{R}^n$ with combinatorial thickness 1, then
the 1-neighborhood of the image has volume at least $N_i^{\frac{n}{n-k} - \epsilon}$.

\end{theorem2.2}

A key fact in the argument of Kolmogorov and Barzdin is that if we map an expander
graph to $\mathbb{R}$, then one of the fibers of the map intersects many edges.
Based on this observation, we define a combinatorial width of k-dimensional complexes.

If $X$ is a k-dimensional simplicial complex, we say that $X$ has combinatorial width
$\ge W$ (over $\mathbb{R}^k$) if any continuous map $F: X \rightarrow \mathbb{R}^k$
has a fiber that intersects $\ge W$ closed k-simplices of $X$.  We illustrate the definition
with two examples.

\vskip10pt

{\bf Example 1.}  If $X$ is an expander graph with $N$ vertices and expansion constant
$h$, then the combinatorial width of $X$ over $\mathbb{R}^1$ is $\ge h N / 2$.
(proof.  pick $y \in \mathbb{R}$ to be the smallest number so that at least half of the vertices of $X$ are mapped to
$( - \infty, y]$.  We let $V_- \subset X$ denote the vertices of $X$ mapped to $(- \infty, y]$ and $V_+$
denote the vertices mapped to $[y, + \infty)$.  By the expansion property, there are at least $Nh / 2$ edges
from $V_-$ to $V_+$.  The image of each of these (closed) edges contains $y$. )

\vskip10pt

{\bf Example 2.} If $X$ is the k-skeleton of the $N$-simplex, then the number of $k$-faces of $X$
is ${N+1 \choose k+1}$.  It is shown in \cite{Gr2} that the
combinatorial width of $X$ over $\mathbb{R}^k$ is at least $c_k {N+1 \choose k+1}$.

\vskip10pt

We can relate the thickness of embeddings and the combinatorial width, basically following
Kolmogorov and Barzdin.

\begin{prop} Suppose that $X$ is a k-dimensional simplicial complex and that each vertex lies in at most $L$ k-simplices. 
Let $W$ denote the combinatorial width of $X$ over $\mathbb{R}^k$.
Let $I$ be an embedding of $X$ into $\mathbb{R}^n$ with combinatorial thickness at least 1.  We also assume that $n \ge 2k + 1$.
Then the 1-neighborhood of $I(X)$ has volume $\gtrsim L^{-\frac{n}{n-k}} W^{\frac{n}{n-k}}$.  
\end{prop}

\begin{proof} Let $X^+$ denote the 1-neighborhood of $I(X)$, and let $V$ be its volume.
By the Falconer slicing estimate, we can map $\mathbb{R}^n$ into $\mathbb{R}^k$ so that each fiber meets $X^+ $ in an area at most $C(k,n) V^{\frac{n-k}{n}}$.  
One of these fibers intersects at least $W$ simplices in $I(X)$.  The number of k-simplices in $X$ that share a vertex with a given k-simplex is at
most $(k+1) L$.  Therefore, we can find $ W L^{-1} (k+1)^{-1}$ simplices in the fiber which do not share any vertices.
Hence the slice must contain $\ge (k+1)^{-1} W L^{-1}$ disjoint unit balls.  So we see that $V^{\frac{n-k}{n}} \gtrsim L^{-1} W$. \end{proof}

Remark: We used the condition $n \ge 2k+1$ in the proof in order to apply the Falconer slicing estimate.  But the
result is probably true for all $k, n$.

If we apply Proposition 2.5 to expanders, we recover the estimate of Kolmogorov and Barzdin.  
Namely, if $X$ is an expander of degree $\lesssim 1$ and expansion constant $h$, we see that the 1-neighborhood of $I(X)$ has volume $\gtrsim |X|^{\frac{n}{n-1}}$.
We can also apply Proposition 2.5 to the k-skeleton of the $N$-simplex.  Each vertex lies in $\lesssim N^k$ k-simplices.  The width
$W$ is $\gtrsim N^{k+1}$.  Therefore, the volume of the 1-neighborhood is at least $N^\frac{n}{n-k}$.  (Is it sharp?)

It would be interesting to know if there are k-dimensional complexes $X$ with locally bounded geometry and large
combinatorial width, generalizing the expander graphs.  More precisely, can we find a sequence of k-dimensional
complexes $X_i$ containing $N_i \rightarrow \infty$ k-simplices so that each vertex of $X_i$ lies in $\le L$
k-simplices and the width of $X_i \ge \alpha N_i$?  (Here $L < \infty$ and $\alpha > 0$ are constants, not depending on
$i$.)  This question was raised in \cite{Gr2}.  It is open.  There are some partial results of different kinds in \cite{FGLNP}
and \cite{Gr2}.

The $X_i$ that we use to prove Theorem 2.2 are the k-skeleta of $D$-dimensional lattices for large $D$.  These
examples are not quite as good as expanders, which leads to the $\epsilon$ in the statement of Theorem 2.2.

Let $X^{k,D}(S)$ denote the k-skeleton of the $D$-dimensional
cubical lattice with lattice spacing 1 restricted to a grid of side length $S$.  This $X^{k,D}(S)$ is a cubical
complex.  The definitions of combinatorial width for cubical complexes 
is analogous to the definition for simplicial complexes.  We will estimate the combinatorial width
of $X^{k,D}(S)$ as a cubical complex.

We will pick a large but fixed $D$ and then study the asymptotics of the spaces $X^{k,D}(S)$ as 
$S \rightarrow \infty$.
The complex $X^{k,D}(S)$ has $\sim c(D,k) S^D$ cubical faces (of all dimensions $0 \le j \le k$).  
Each vertex lies in $\le L(D, k)$ faces.

\begin{prop} The combinatorial width of $X^{k,D}(S)$ over $\mathbb{R}^k$ is $\ge c(D,k) S^{D-k}$.
\end{prop}

Using Proposition 2.6, we can quickly finish the proof of Theorem 2.2.

\begin{proof}[Proof of Theorem 2.2]
We fix $D$ sufficiently large, and then study the sequence of complexes $X^{k,D}(S)$ as 
$S \rightarrow \infty$.  Since the statement of Theorem 2.2 involves simplicial complexes,
we subdivide each cubical complex $X^{k,D}(S)$ to get a simplicial complex $X^{k,D}_{simp}(S)$.  

Let us write $W(S)$ to abbreviate the combinatorial width of $X^{k,D}_{simp}(S)$, and let us write
$N(S)$ to abbreviate the number of simplices in $X^{k,D}_{simp}(S)$.

The combinatorial width of  $X^{k,D}_{simp}(S)$ is at least as big as the combinatorial width of
$X^{k,D}(S)$.  By Proposition 2.6, $W(S) \ge c(D,k) S^{D-k}$.

On the other hand, $X^{k,D}_{simp}(S)$ has $\le C(D,k) S^D$ faces, so $N(S) \le C(D,k) S^D$.
Therefore, $W(S) \ge c(D,k) N(S)^\frac{D-k}{D}$.  We choose $D = D(\epsilon,k)$ sufficiently large so that
$W(S) \ge c(\epsilon,k) N(S)^{1- \epsilon}$.

The local geometry of $X^{k,D}_{simp}(S)$ depends only on $k$ and $D$.  In particular, each vertex of
$X^{k,D}(S)$ lies in $\le L(D,k)$ simplices of $X^{k,D}_{simp}(S)$.

Now we apply Proposition 2.5.  If $I$ is an embedding from $X^{k,D}_{simp}(S)$ into
$\mathbb{R}^n$ with combinatorial thickness $\ge 1$, then the 1-neighborhood of the image of $I$
has volume at least $c(D,k) W(S)^{\frac{n}{n-k}} \ge c(D,k) N(S)^{\frac{n}{n-k} - \epsilon}$.  \end{proof}

In the remainder of Section 2.2, we explain the proof of Proposition 2.6, our estimate for the combinatorial
width of $X^{k,D}(S)$.  This combinatorial width estimate is a combinatorial analogue
of a purely geometric estimate for the width.

\newtheorem*{widcube}{Width of the cube}

\begin{widcube} If $F: [0,S]^D \rightarrow \mathbb{R}^k$ is a generic smooth map, then
one of the fibers of $F$ has $(D-k)$-volume at least $c(D,k) S^{D-k}$.
\end{widcube}

(This geometric theorem was proven by Almgren using minimal surface theory.  In \cite{Gr3},
there is a more elementary proof using isoperimetric inequalities.  See \cite{Gu} for more details.)
In \cite{Gr2}, the isoperimetric proof was adapted to the combinatorial
setting.  Instead of using classical isoperimetric inequalities, this method uses filling
inequalities for cochains.  Let's recall what a filling inequality is.

Let $X$ denote a k-dimensional polyhedral complex.  We let $C^*(X, \mathbb{Z}_2)$ denote
the polyhedral cochains of $X$.  For a polyhedral $j$-cochain $\alpha$, we define $\| \alpha \|_{1}$
to be the number of j-faces in the support of $\alpha$.  Now we define the constant $Fill(j)$
to be the smallest number so that any coboundary $\alpha \in C^j(X, \mathbb{Z}_2)$ has
a ``filling" $\beta \in C^{j-1}(X , \mathbb{Z}_2)$ so that $d \beta = \alpha$ and
$\| \beta \|_1 \le Fill(j) \| \alpha \|_1$.  (For a finite complex $X$, the constant $Fill(j)$
is finite, but for an infinite complex, it could be $+ \infty$.)

The combinatorial width of a complex can be controlled in terms of its filling inequalities
by the following theorem.

\newtheorem*{widfill}{Width and filling}

\begin{widfill} (\cite{Gr2}) Suppose that $X$ is a k-dimensional simplicial complex or cubical complex
with $N$ vertices.  Suppose $X$ is connected and $H^j(X, \mathbb{Z}_2) = 0$ for $1 \le j \le k-1$.
Let $W$ be the combinatorial width of $X$ over $\mathbb{R}^k$.  Then the following inequality
holds.

$$ W \ge c_k \left[ \prod_{j=1}^k Fill(j) \right] ^{-1} N. \eqno{(**)}$$

\end{widfill}

(For explanation and proof see \cite{Gr2} and \cite{D}.)

The spaces $X^{k,D}(S)$ are connected and have vanishing homotopy and cohomology groups in
dimensions $1 \le j \le k-1$, so this inequality applies to them.  
We now bound the filling constants of the spaces $X^{k,D}(S)$.

\begin{prop} If $X = X^{k,D}(S)$ defined above, then the filling constant $Fill(j)$ is
$\le C(D,k) S$.
\end{prop}

Plugging Proposition 2.7 into the Width-and-filling inequality (**), we see that
the combinatorial width of $X^{k,D}(S)$ is at least $c(D,k) S^{D-k}$ as 
desired.

If $X^k$ is the k-skeleton of a D-manifold with a polyhedral structure, then we can apply Poincare duality
and relate $Fill(j)$ to a homological filling for $(D-j)$-cycles in the manifold.  Let $Y$ denote the
dual polyhedral structrure.  Then $Fill(j)$ can equivalently be defined as follows:
$Fill(j)$  is the smallest number so that every boundary $a \in C_{D-j}(Y, \mathbb{Z}_2)$ has
 a filling $b \in C_{D-j+1}(Y, \mathbb{Z}_2)$ with $\partial b = a$ and $Vol(b) \le Fill(j) Vol(a)$.
 Here $Vol(b)$ denotes the number of faces in the support of $b$.  This kind of filling inequality
 is a combinatorial version of an isoperimetric inequality.

The complexes $X^{k,D}(S)$ are the k-skeleta of a polyhedral structure on the D-cube.  The dual
polyhedral structure, $Y$, is also a cubical grid.  It has one vertex at the center of each
D-cell in the original grid, and so on.  A j-dimensional cocycle in $X^{j,D}(S)$ corresponds to
a (D-j)-dimensional relative cycle in $Y$.

So after applying Poincare duality, the filling constant bound $Fill(j) \le C(D,k) S$ for $X^{k,D}(S)$
is equivalent to the following estimate.

\begin{prop} Let $Y$ denote a cubical grid in $D$-dimensional space with side length $S$.  Suppose
that $a$ is a relative (D-j)-cycle in the grid consisting of $|a|$ faces.  Then $a$ is the (relative) boundary of
a (D-j+1)-chain $b$ in the grid consisting of $\le c(D,k) S |a|$ faces.
\end{prop}
 
This proposition is a combinatorial version of the standard Federer-Fleming pushout argument.
(See \cite{Gu2} for the details of the standard argument.)  The combinatorial version is an
easy consequence of the standard pushout argument and the deformation theorem.
 
\proof  Let $x$ be a random point (not necessarily a vertex) in the cube $Y$.  We 
fill $a$ by a geometric (but not cubical) chain by pushing out from the point $x$.  This is defined
in the following way.  Let $C(a)$ denote the cone consisting of all the line segments from $x$
to points on $a$.  Let $C^+(a)$ denote the infinite cone, consisting of all the rays from $x$ through
a point of $a$.  We can fill $a$ by the relative chain given by $[C^+(a) - C(a)] \cap Y$.

By an averaging argument invented by Federer and Fleming, one sees that the average volume of
this chain, as $x$ varies in the cube $Y$, is bounded by $C(D,j) S |a|$.  Let $b_0$ be one of these
chains with at most average volume.

The chain $b_0$ is not a cubical chain.  We finish the argument by approximating $b_0$ with a cubical
chain $b$.  The deformation theorem allows us to approximate $b_0$ by a cubical chain $b$ so
that $\partial b = \partial b_0 = a$, and with $Vol(b) \le C(D,j) Vol(b_0)$.

This chain $b$ obeys the conclusion of the proposition. \endproof

\subsection{Open questions: triangulated manifolds}

There are many open questions about the case of triangulated manifolds.  
Suppose $X$ is a triangulated k-manifold with $\lesssim 1$ simplices
containing any vertex and with $N$ simplices total.  If $n \ge 2k+1$, then 
we can use the embedding construction from Section 2.1 to
embed $X$ with thickness 1 into
$B^n(R)$ for $R \lesssim C_\epsilon N^{\frac{1}{n-k} + \epsilon}$.  We don't know
whether this is sharp. 

The combinatorial width method used in Section 2.2 does not work directly for manifolds.  In
Section 2.9f of \cite{Gr2}, it is proven that the combinatorial width of a k-manifold $X^k$
over $\mathbb{R}^k$ is always $\le 4k$.
One may still get some estimates by considering lower-dimensional skeleta of the manifold.  A thick
embedding of $X^k$ into $\mathbb{R}^n$ restricts to a thick embedding
of any skeleton of $X^k$.

Now we consider some examples.

\vskip10pt

{\bf Example 0.} If $k=1$, then a triangulated closed k-manifold is just a union of circles, each circle
divided into some number of edges.  If the total number of edges is $N$, then $X^1$ embeds with
combinatorial thickness 1 into $B^n(R)$ for $R \sim N^{1/n}$ in any dimension $n \ge 2$.  The case
$k=1$ is especially simple, but it may still suggest that manifolds are easier to embed than general
complexes.

\vskip10pt

{\bf Example 1.}  Consider a family of arithmetic hyperbolic surfaces, 
$\Sigma_i$ given as coverings of a surface $\Sigma_0$.  We let $G_i \rightarrow
+ \infty$ denote the genus of $\Sigma_i$.  These surfaces have natural triangulations, given by picking
a triangulation of $\Sigma_0$ and then lifting it to the $\Sigma_i$.  What is the smallest radius
so that the triangulated manifold $\Sigma_i$ embeds into $B^n(R)$ with combinatorial thickness
1?  Theorem 2.1 guarantees that such embedding exists with $R \sim G_i^{\frac{1}{n-2}}$ for $n \ge 5$.

On the other hand,  the 1-skeleton of our triangulation of $\Sigma_i$ is an
expander graph with $\sim G_i$ vertices.  The expander properties of arithmetic hyperbolic
manifolds are discussed more in the appendix.  An embedding of $\Sigma_i$ with combinatorial
thickness 1 restricts to give an embedding of the 1-skeleton with combinatorial thickness at least
1.  Therefore, the radius $R$ above must be $\gtrsim G_i^{\frac{1}{n-1}}$.  Moreover, there is
an actual embedding of the 1-skeleton into $B^n(R)$ with combinatorial thickness 1 and 
$R \sim G_i^{\frac{1}{n-1}}$ for $n \ge 3$.  Can we extend this embedding to all of $\Sigma_i$?

(One should remember here that the combinatorial thickness of an embedding $\Sigma_i \rightarrow
\mathbb{R}^n$ depends on the choice of triangulation.  Two different triangulations of the same surface
can have very different properties.  One may find other triangulations of a surface of genus $G_i$ that embed
with combinatorial thickness 1 into a ball $B^n(R)$ of radius $R \lesssim G_i^{1/n}$.  In the next
section, we will study the retraction thickness of embeddings $\Sigma \rightarrow \mathbb{R}^n$.  The retraction
thickness does not depend on a triangulation of $\Sigma$.)

\vskip10pt

{\bf Example 2.}  Let $k = 2j$.  Begin with the j-dimensional skeleton
of a $D$-dimensional lattice with grid size $S$: $X^{j,D}(S)$.  We fix $D$
large, and then send $S \rightarrow \infty$.  This complex
has $\sim S^D$ simplices.  Embed this
complex into $\mathbb{R}^{k+1}$.  Take a small tubular neighborhood,
and let $X^k$ be the boundary of this neighborhood, which is a manifold.
We can triangulate $X^k$ so that $X^{j,D}(S)$ appears inside the j-skeleton
of the triangulation, and also the triangulation has bounded local degree
and $N \sim S^D$ simplices.  By the lower bound in the last subsection,
a 1-thick embedding of $X^k$ into $\mathbb{R}^n$ has 1-neighborhood
with volume at least $N^{\frac{n}{n-j}}$.

\vskip10pt

If $X^k$ is a simplicial complex, we define the combinatorially-thick embedding radius of 
$X$ in $\mathbb{R}^n$ to be the smallest $R$ so that $X$ embeds
with combinatorial thickness 1 into $B^n(R)$.

\vskip10pt

{\bf Question 1.} If $X^k$ is a triangulated manifold with $N$ simplices
and each vertex in $\lesssim 1$ simplices, what can we say about the
combinatorially-thick embedding radius of $X^k$?  
Is any $X^k$ significantly harder to embed than
Example 2 above?

\vskip10pt

{\bf Question 2.} If $X^k$ is a triangulated manifold with $N$ simplices
and each vertex in $\lesssim 1$ simplices, what can we say about the
combinatorial thickness of $X^j$ over $\mathbb{R}^j$?  Here $1 \le j \le k$,
and $X^j$ denotes the j-skeleton of $X$.

\vskip10pt

If $1 \le j \le k/2$, then Example 2 shows that $X^j$ may have combinatorial
thickness $\gtrsim N^{1-\epsilon}$, which is nearly optimal.  If $j = k$, then
the result of Sec 2.9f of \cite{Gr2} says that the combinatorial width is
$\le 4k $ regardless of $N$.  For $k/2 < j < k$, the situation is unclear.
This $X^j$ may be the j-skeleton of a k-dimensional grid, and the results
in Section 2.2 show that the combinatorial width of such $X^j$ over
$\mathbb{R}^j$ is $\gtrsim N^{\frac{k-j}{k}}$.  

\vskip10pt

{\bf Question 3.} If $X^k$ is a triangulated manifold with $N$ simplices
and each vertex in $\lesssim 1$ simplices, what can we say about the
filling constants $Fill(j)$ of $X^k$?

\vskip10pt

If $1 \le j \le k/2$, then Example 2 shows that $Fill(j)$ may be $\lesssim N^\epsilon$.
On the other hand, for any complex $Fill(j) \gtrsim 1$, so this upper bound is
nearly sharp.  For $j > k/2$, the situation is unclear.  The $X^j$ may be
the j-skeleton of a cubical grid of side length $N^{1/k}$.  
In this case, the results of Section 2.2 show that $Fill(j) \lesssim N^{1/k}$.  

Also, if $X$ is connected and $H^j(X, \mathbb{Z}_2) = 0$ for all $1 \le j \le
k-1$, then the width-and-filling inequality cited in Section 2.2 applies.  Since
the combinatorial width of $X$ over $\mathbb{R}^k$ is only $4k$, we
get the following estimate:

$$N \le C(k) \prod_{j=1}^k Fill(j). $$

\vskip10pt

Finally, if $X$ is a k-manifold, it embeds into $\mathbb{R}^{2k}$ by 
the Whitney embedding theorem.  It would be interesting to prove quantitative
estimates for such an embedding.

\section{Estimates for retraction thickness}

Suppose that $X$ is a CW-complex embedded in $\mathbb{R}^n$.  We say that the retraction thickness of
$X$ is at least $T$ if the $T$-neighborhood of $X$ retracts to $X$.  In this section, we study how the retraction
thickness relates to the topology of $X$.  Suppose for example that $X$ is embedded in $\mathbb{R}^n$ with
retraction thickness 1 and that the 1-neighborhood of $X$ has volume $V$.  Based on this information, what
can we conclude about the topology of $X$?

\subsection{Constructing emdeddings}

The quantitative general position arguments from Section 2.1 can also be used to construct embeddings with controlled retraction thickness.

\newtheorem*{theorem3.1}{Theorem 3.1}

\begin{theorem3.1} If $X^k$ is a k-dimensional simplicial complex with $N$ simplices and each vertex lies in
$\le L$ simplices, and if $n \ge 2 k + 1$, then there is an embedding with retraction thickness 1 from $X$ 
into $B^n(R)$ for $R \le C(n,L, \epsilon) N^{\frac{1}{n-k} + \epsilon}$.  
\end{theorem3.1}

The proof of Theorem 3.1 is a technical modification of the proof of Theorem 2.1.  We give the proof in Section 3.4.

For combinatorial thickness, Theorem 2.1 is extremely sharp.  But we don't know how sharp Theorem 3.1 is.  When $k=1$, we will see that Theorem 3.1 is far from sharp.  We don't know what happens when $k \ge 2$.

When we first defined retraction thickness, we expected it to be a minor variation of the combinatorial thickness
discussed in Section 2.  But it turns out to be quite different.  For example, an expander with $N$ vertices 
embeds with combinatorial thickness 1 into $B^3(R)$ only for $R \gtrsim N^{1/2}$, but it embeds with retraction thickness 1 into 
$B^3(R)$ for $R \sim N^{1/3}$.  From the point of view of combinatorial thickness, an expander is much more complicated
than a grid (with the same number of vertices).  But from the point of view of retraction thickness, an expander and a grid are
basically equivalent.  The underlying reason is that an expander is homotopy equivalent to a grid with roughly the same number
of vertices.  The significance of homotopy equivalence appears in the following proposition.

\begin{prop} Let $X^k$ and $Y^k$ be homotopy equivalent CW-complexes of dimension $k$.  Suppose that $n \ge 2k + 1$,
and suppose that $Y$ embeds into $\mathbb{R}^n$ with retraction thickness $T$.  Then $X$ embeds into $\mathbb{R}^n$
with retraction thickness at least $T - \epsilon$ for any $\epsilon > 0$.
\end{prop}

\proof We consider $Y$ as a subset of $\mathbb{R}^n$.  Let $\Psi: X \rightarrow Y$ be a homotopy equivalence.  We can think of
$\Psi$ as a map $X \rightarrow \mathbb{R}^n$.  Now $\Psi$ may not be an embedding, but since $n \ge 2k + 1$, we can 
perturb $\Psi$ slightly to an embedding $\Psi'$.  Since the perturbation is slight, we may arrange that the $(T - \epsilon)$-neighborhood
of $\Psi'(X)$ lies inside of the $T$-neighborhood of $Y$.

We know that there is a retraction $R$ from the $T$-neighborhood of $Y$ to $Y$.  Since $\Psi$ is a homotopy equivalence, we also
know that there is a map $\Phi: Y \rightarrow X$ so that $\Phi \circ \Psi: X \rightarrow X$ is homotopic to the identity.  We consider
the map $\Phi \circ R$ which goes from the $(T-\epsilon)$-neighborhood of $\Psi'(X)$ to $X$.  We claim that this map is a homotopy
retraction.  In other words, when we restrict the map to $X = \Psi'(X)$, we get a map homotopic to the identity.  To see this,
notice that $\Psi'$ is homotopic to $\Psi$ inside the $(T-\epsilon)$ neighborhood.  So the restriction of our map to $X$ is homotopic
to $\Phi \circ R \circ \Psi$.  But $\Psi$ maps $X$ into $Y$, and $R$ is the identity on $Y$, so $\Phi \circ R \circ \Psi$ is equal
to $\Phi \circ \Psi$ which is homotopic to the identity.

We have now proven that the $(T - \epsilon)$-neighborhood of $\Psi'(X)$ homotopy retracts to $\Psi'(X)$ by the map $\Phi \circ R$.
If we restrict this map to $\Psi'(X)$, then it is homotopic to the identity.  By the homotopy extension theorem, we can extend this homotopy
to get a retraction from the $(T - \epsilon)$-neighborhood of $\Psi'(X)$ to $\Psi'(X)$.  In other words, $X$ is embedded in 
$\mathbb{R}^n$ with retraction thickness at least $T - \epsilon$. \endproof

We pause here to recall some results we use about homotopy retractions and the homotopy extension property.  First, we recall the definition
of a homotopy retraction.  If $A \subset B$, then a map $F: B \rightarrow A$ is a homotopy retraction if the restriction $F: A \rightarrow A$
is homotopic to the identity.  A retraction is of course a special case of a homotopy retraction.  On the other hand, if $A$ is a finite
simplicial complex or CW complex topologically embedded in $\mathbb{R}^n$, and $B$ is a neighborhood of $A$, then the pair $(B,A)$ has the homotopy
extension property.  This means that if $g: B \rightarrow Y$ is a continuous map to any space $Y$, and $h: A \times [0,1] \rightarrow Y$
is a homotopy with $h(a,0) = g(a)$ for all $a \in A$, then the homotopy $h$ can be extended to $G: B \times [0,1] \rightarrow Y$ with
$G(b,0) = g(b)$ for all $b \in B$.  (See for example sections 2.9-2.10 in \cite{NR}.)  In particular, any homotopy retraction $B \rightarrow A$
may be homotoped to a genuine retraction $B \rightarrow A$.  When we have to construct a retraction, it can be easier to first construct a
homotopy retraction and then use this method to move it to a real retraction.  We did this in the proof of Proposition 3.1, and we will do it
again when we prove Theorem 3.1 in Section 3.4.  The theorem about the homotopy extension property cited above has a fairly
short proof, but we could also avoid using it by talking throughout about homotopy retractions and ``homotopy retraction thickness".  
Everything we prove about retraction thickness below applies just as well with homotopy retractions in place of retractions.

As a corollary of Proposition 3.1, we see that expanders embed with retraction thickness 1 into rather small balls.

\begin{prop} Let $\Gamma$ be any graph with degree $\lesssim 1$ and $N$ vertices.  Then $\Gamma$ may
be embedded with retraction thickness 1 into $B^3(R)$ for $R \lesssim N^{1/3}$.
\end{prop}

\begin{proof} It's straightforward to reduce to the case that $\Gamma$ is connected.  
If $\Gamma$ is connected, then it is homotopy equivalent to any other graph with the same rank of $H^1$.
In particular, $\Gamma$ is homotopy equivalent to a subgraph of a grid with $\lesssim N$ vertices.
As we saw in Section 1, such a grid can be embedded with combinatorial thickness 2 into a ball of
radius $\lesssim N^{1/3}$.  It's straightforward to check that the embedding we constructed in Section 1 also
has retraction thickness at least 2.  By Proposition 3.1, we get an embedding of $\Gamma$ with
retraction thickness 1 into a ball of radius $\sim R^{1/3}$. \end{proof}

Proposition 3.2 turns out to be sharp, as we will see in the next section.

\subsection{Estimates for rank of homology and simplicial volume}

Suppose that $X$ is a CW-complex embedded in $\mathbb{R}^n$ with retraction thickness $T$.
Let $V$ denote the volume of the $T$-neighborhood of $X$.  Then $V T^{-n}$ is a scale-invariant
quantity that describes the geometric complexity of the embedding.  As we will see, this geometrical
complexity controls some aspects of the homotopical complexity of $X$.  In this subsection, we
show that $V T^{-n}$ controls the homology groups of $X$ and the simplicial volume of $X$.

\begin{prop} Suppose $X$ is a CW-complex embedded in $\mathbb{R}^n$ with retraction thickness
$T$.   Let $V$ be the volume of the $T$-neighborhood of $X$.  Then
the following inequality holds.

$$ Rank H_*(X) \lesssim V T^{-n}. $$

\end{prop}

\begin{proof} We write $N_W(X)$ to denote the $W$-neighborhood of $X$.

Let $p_1, p_2, ...$ be a maximal set of $T/2$-separated points in $N_{T/2}(X)$.
The balls $B(p_i, T/2)$ cover $N_{T/2}(X)$.  Since the balls $B(p_i, T/4)$ are disjoint
and lie within $N_T(X)$, we see that the number of points is $\lesssim V T^{-n}$.
We let $U := \cup_i B(p_i, T/2)$.  Since $X \subset U \subset N_T(X)$,
we see that $U$ retracts to $X$.

The balls $B(p_i, T/2)$ give an open cover of $U$,
and we let $N$ denote the nerve of this cover.  
We have a natural map to the nerve $\Phi: U \rightarrow N$.  
Since the intersection of any set of balls is convex and hence contractible, the map $\Phi$ is a homotopy
equivalence of $U$ and $N$.

With this setup, we show that the identity map from $X$ to $X$ factors through $N$ up to homotopy.  We can map $X$ to $N$ by 
restricting $\Phi$ to $X$.  Now $\Phi$ has a homotopy inverse $\Psi: N \rightarrow U$.  We let $R$
denote the retraction from $U$ to $X$.  Now $\Psi \circ \Phi: U \rightarrow U$ is homotopic to the identity.
Therefore, $R \circ \Psi \circ \Phi: X \rightarrow X$ is also homotopic to the identity.

As a consequence, $Rank H_*(X) \le Rank H_*(N)$.

Our cover of $U$ involves $\lesssim V T^{-n}$ balls $B(p_i, T/2)$.  Since the balls $B(p_i, T/4)$ are disjoint, the multiplicity
of the cover is $\lesssim 1$.  Therefore the nerve $N$
contains at most $\lesssim V T^{-n}$ simplices.  Therefore, $Rank H_*(N) \lesssim V T^{-n}$.

\end{proof}

In particular, if $\Gamma$ is a grid or expander with $N$ vertices then $Rank H_1(\Gamma) \sim N$.  So if
$\Gamma$ embeds with retraction thickness 1 into $B^3(R)$, we see that $R \gtrsim N^{1/3}$.  So we see
that Proposition 3.2 is sharp.

The next proposition concerns simplicial norms modulo $p$.  If $a$ lies in $\mathbb{Z}_p$ (the integers modulo $p$),
we define $|a|$ to be zero if $a = 0$ and one if $a \not=0$.  If $h \in H_k(X, \mathbb{Z}_p)$, we define
the simplicial norm of $h$ to be the infimum of $\sum |a_j|$ over all cycles 
$\sum_j a_j f_j$ which represent $h$.  (In this formula, the $f_j$ denote maps from the k-simplex to $X$ and $a_j \in \mathbb{Z}_p$.)  
We denote the simplicial norm of $h$ by $\| h \|_{\triangle}$.  Simplicial norms over $\mathbb{R}$ are better known, but simplicial norms over $\mathbb{Z}_p$ share some
of their good properties.  In particular, the simplex straightening argument shows that for a complete hyperbolic manifold
$(M, hyp)$, for any $h \in H_k(M, \mathbb{Z}_p)$ with $k \ge 2$, we have $\| h \|_{\triangle} \gtrsim mass(h)$.
In particular, if $M$ is closed, the simplicial norm of the fundamental class $[M]$ is $\gtrsim Vol(M, hyp)$.  (See \cite{MT}
and the appendix.)

\begin{prop} Suppose $X$ is a CW-complex embedded in $\mathbb{R}^n$ with retraction
thickness $T$.  Let $V$ be the volume of the $T$-neighborhood of $X$.  
Let $h \in H_k(X, \mathbb{Z}_p)$ be any homology class.  Then the following inequality holds.

$$ \| h \|_{\triangle} \lesssim V T^{-n}. $$

\end{prop}

\begin{proof} As in the proof of Proposition 3.3, we define the nerve $N$ and the map $\Phi: X \rightarrow N$.  
As we saw in the proof of Proposition 3.3, the identity map from $X$ to itself factors through the map $\Phi: X \rightarrow N$,
up to homotopy.

If $F: X \rightarrow Y$ is any continuous map, and $h$ is any homology class in $H_k(X, \mathbb{Z}_p)$, 
it follows from the definition of simplicial norm that $\| F_*(h) \|_{\triangle} \le \| h \|_{\triangle}$.  Since the identity map factors
through $\Phi$, it follows that $\| \Phi(h) \|_{\triangle} = \| h \|_{\triangle}$.

But we can represent $\Phi(h)$ by a sum $\sum a_j \Delta_j$ where $a_j \in \mathbb{Z}_p$ and $\Delta_j$ are
simplices of the nerve $N$.  Hence $\| h \|_{\triangle} = \| \Phi(h) \|_{\triangle}$ is at most the number of simplices in $N$.
As we saw in the proof of Proposition 3.3, the number of simplices in $N$ is $\lesssim V T^{-N}$.  \end{proof}

(Here are some further questions about bounding the homotopical complexity of $X$ in terms of
$V T^{-n}$.  Can we bound the torsion in $H_*(X, \mathbb{Z})$ in terms of $V T^{-n}$?  For example,
if $X_D$ is the CW complex $S^d \cup_f B^{d+1}$, where the attaching map $f: S^d \rightarrow S^d$ has
degree $D$, and $X_D$ is embedded in $\mathbb{R}^n$, how small can $V T^{-n}$ be ?  
Also, can we bound the real simplicial volume of $X$ in terms of $V T^{-n}$?)

To finish Section 3.2, we apply these propositions to some examples.  We get good estimates for the
retraction thickness for k-skeleta of grids, for 2-dimensional manifolds, and for some 3-dimensional manifolds.

\vskip10pt

{\bf Example 1.} Let $X^{k,D}(S)$ denote the k-skeleton of a $D$-dimensional grid of side length $S$.  We fix $k, D$ and consider
asymptotics as $S \rightarrow \infty$.  For this example, the retraction thickness behaves very differently from the combinatorial thickness.
We let $N$ denote the number of simplices in $X^{k,D}(S)$.  If $X^{k,D}(S)$ is embedded with combinatorial thickness 1 into $B^n(R)$, then we showed in Section 2.2 that $R \gtrsim N^{\frac{1}{n-k} \frac{D-k}{D}} \sim N^{\frac{1}{n-k} - \epsilon}$.  But for any $n \ge 2k+1$, we will
embed $X^{k,D}(S)$ with retraction thickness 1 into $B^n(R)$ for $R \sim N^{1/n}$.

To construct the embedding, let $Y$ be 
the k-skeleton of an n-dimensional grid which is homotopy equivalent to $X$.  Such a $Y$ has $\sim N$ cells, and
it comes embedded with retraction thickness $\gtrsim 1$ into a ball $B^n(R)$ with $R \sim N^{1/n}$.  Since $X$ and $Y$
are homotopy equivalent and $n \ge 2k+1$, Proposition 3.1 implies that $X$ embeds into the same $B^n(R)$ with
retraction thickness $\gtrsim 1$.

This construction is basically sharp.  The complex $X^{k,D}(S)$ has $N \sim S^D$ cells, and $H^k (X) $ has rank $\sim S^D \sim N$.
By Proposition 3.3, if $X$ embeds with retraction thickness 1 into $B^n(R)$, then $R \gtrsim N^{1/n}$.  

\vskip10pt

{\bf Example 2.} Let $\Sigma$ be a closed surface of genus $g$.  For each $n \ge 3$, 
it's straightforward to embed $\Sigma$ with retraction thickness 2 into
an n-dimensional cylinder with radius 10 and length $10 g$: $B^{n-1}(10) \times [0, 10 g]$.  
This cylinder C-bilipschitz embeds into $B^n(R)$ for $R \sim g^{1/n}$.  Hence
$\Sigma$ embeds with retraction thickness 1 into $B^n(R)$ for $R \sim g^{1/n}$.  
Since $H_1(\Sigma)$ has rank $2g$, this radius cannot be improved.

\vskip10pt

{\bf Example 3.} Let $X$ be a closed k-dimensional hyperbolic manifold with volume $V$.
In particular, the simplicial volume of $X$ is $\gtrsim V$.  Therefore, if $X$ is embedded with
retraction thickness 1 into $\mathbb{R}^n$, then the 1-neighborhood of $X$ has volume
$\gtrsim V$.  If $X$ is embedded with retraction thickness 1 into $B^n(R)$, then 
$R \gtrsim V^{1/n}$.  

This estimate is sharp for cyclic coverings.  Suppose that $X_1$ is a closed
hyperbolic k-manifold and that there is a surjective homomorphism $\pi_1(X_1) \rightarrow \mathbb{Z}$.
Then we can construct corresponding cyclic coverings $X_D \rightarrow X_1$ of degree $D$.  The hyperbolic
volume of $X_D$ is $\sim D$.  If 
$n \ge 2k+1$, then this $X_D$ can be embedded with retraction thickness 1 into a solid torus $S^1(L) \times B^{n-1}(r)$ with
radius $r \lesssim 1$ and $L \sim D$.  This solid torus in turn admits a $C$-bilipschitz
embedding into $B^n(R)$ for $R \sim D^{1/n}$.  Hence $X_D$ can be embedded with retraction thickness
1 into $B^n(R)$ for $R \sim D^{1/n}$.  

\subsection{Retraction thickness for arithmetic hyperbolic manifolds}

We will prove a stronger bound on the retraction thickness for arithmetic hyperbolic manifolds of
dimension $k \ge 3$.  In the appendix, we give a review of arithmetic hyperbolic manifolds.  Here, we quickly
mention the facts we need for our argument.  In each dimension $k \ge 2$, arithmetic covers can be used to construct 
a sequence of closed hyperbolic manifolds $X_i$ with volume $V_{hyp}(X_i)
\rightarrow + \infty$, and with uniformly bounded Cheeger isoperimetric constant $h(X_i) \ge c > 0$ for all $i$.
Moreover, each $X_i$ may be triangulated with $N_i \sim V_{hyp}(X_i)$ simplices.  The following theorem
shows that it is difficult to embed such arithmetic manifolds into $\mathbb{R}^n$ with a given retraction thickness.

\newtheorem*{theorem3.2}{Theorem 3.2}
\begin{theorem3.2} Suppose that $X^k$ is a closed hyperbolic manifold with volume $V_{hyp}$ and Cheeger isoperimetric
constant $h$, with dimension $k \ge 3$.  If
$X$ is embedded in $\mathbb{R}^n$ with retraction thickness $T$, and the $T$-neighborhood of $X$ in $\mathbb{R}^n$
has volume $V$, then the following inequality holds.

$$ h^{\frac{n}{n-1}} V_{hyp}^{\frac{n}{n-1}} \lesssim V T^{-n}. $$
\end{theorem3.2}

Before turning to the proof, we make some comments about the theorem.  
The most interesting case is when $X$ is an arithmetic hyperbolic manifold described above.  Such a manifold
$X$ can be triangulated with $N$ simplices, and has $V_{hyp} \sim N$ and $h \sim 1$.  Therefore, if $X$ is
embedded into $\mathbb{R}^n$ with retraction thickness 1, the 1-neighborhood of $X$ has volume $\gtrsim
N^{\frac{n}{n-1}}$.  This estimate is significantly stronger than what we get by applying the estimate for
$Rank H_*(X)$ in Proposition 3.3 or the estimate for simplicial volume in Proposition 3.4.  
Since $X$ is triangulated with $N$ simplices,
$Rank H_*(X) \lesssim N$ and the simplicial norm of $X$ is $\le N$.  Using the rank of $H_*(X)$ and the simplicial
norms of $X$, we can only prove that the volume of the 1-neighborhood of $X$ is $\gtrsim N$.  
The moral of Theorem 3.2 is that the hyperbolic manifolds with Cheeger constant $\sim 1$ are more topologically 
complicated than other hyperbolic manifolds of the same volume (such as cyclic coverings). 

On the other hand, there is a significant gap between the bounds in Theorem 3.1 and Theorem 3.2.  For example,
consider embedding a 3-dimensional arithmetic hyperbolic manifold $X^3$ into $B^7(R)$.  What is the smallest
$R$ so that $X^3$ embeds in $B^7(R)$ with retraction thickness 1?  Theorem 3.1 says that this radius is
$\lesssim N^{1/4 + \epsilon}$.  Theorem 3.2 says that this radius is $\gtrsim N^{1/6}$. 

Theorem 3.2 is somewhat analogous to the theorem about combinatorial thickness of expander graphs.  To see the analogy,
let us look at the case $k=3$, and consider an arithmetic hyperbolic 3-manifold $X^3$.  A key fact
about expander graphs is that any map from an expander graph to $\mathbb{R}$ has a ``complicated" level set
which intersects many edges.  The key fact about arithmetic hyperbolic 3-manifolds is that any map from $X^3$ to 
$\mathbb{R}$ has a ``complicated" level set which is a surface of large genus.  See the appendix for more information,
and some history of this fact.

Using this key fact, we may sketch the main point of our proof.  Suppose that $X$ is an arithmetic hyperbolic 3-manifold
triangulated with $N$ simplices embedded into $\mathbb{R}^n$ with retraction thickness 1, and let $V$ denote
the volume of $N_1(X)$, the 1-neighborhood of $X$.  By the Falconer slicing theorem, we may slice $\mathbb{R}^n$
by parallel hyperplanes so that each hyperplane intersects $N_1(X)$ in surface area $\lesssim V^{\frac{n-1}{n}}$.  (The Falconer
slicing estimate is recalled in Section 1.1.)  By
the key fact in the last paragraph, one of the hyperplanes intersects $X$ in a surface $\Sigma$ of genus $\gtrsim N$.
Let us call this hyperplane $P$.  We know that $\Sigma$ has genus $\gtrsim N$ and we know that the 1-neighborhood
of $\Sigma$ in $P$ has volume $\lesssim V^{\frac{n-1}{n}}$.  Suppose for a second that we knew that the retraction thickness
of $\Sigma$ in $P$ was $\ge 1$.  Then Proposition 3.3 would imply that $N \lesssim V^{\frac{n-1}{n}}$ - the estimate we would 
like to prove.  Morally, we think that this is the main idea of the proof.  But strictly speaking, the 1-neighborhood of $\Sigma$
may not retract to $\Sigma$ - it only needs to retract into $X$.  In the proof below we get around this technical difficulty
by working with the nerve of a covering by balls, and by keeping track of many slices and how they fit together.

\begin{proof}[Proof of Theorem 3.2]  By scaling, we may assume that $T = 1$.

We let $N_W(X) \subset \mathbb{R}^n$ denote the $W$-neighborhood of $X \subset \mathbb{R}^n$.  
Choose a maximal $1/4$-separated set of points in $N_{1/2}(X)$: $p_1, p_2, ...$  Let $U$ be the 
open set $U := \cup_{i} B(p_i, 1/4)$.  We have $X \subset N_{1/2}(X) \subset U \subset N_{3/4}(X)$.
In particular, $U$ retracts to $X$.

Our set $U$ is covered by balls $U := \cup_i B(p_i, 1/4)$.  We let $N$ denote the nerve of this cover.
Since an intersection of balls is convex and hence contractible, the nerve $N$ is homotopy equivalent
to the cover $U$.  We let $\Phi: U \rightarrow N$ denote a map subordinate to the cover, and we let
$\Psi: N \rightarrow U$ be a homotopy inverse of $\Phi$.  (In particular, $\Psi \circ \Phi: U \rightarrow
U$ is homotopic to the identity.)

Next we apply the Falconer slicing inequality to the set $N_1(X)$ which has volume $V$.  (The Falconer
slicing inequality is recalled in Section 1.)  According to the slicing inequality, we can rotate our coordinate
frame so that for every $t \in \mathbb{R}$, the (n-1)-dimensional volume of the ``slice" $N_1(X) \cap \{ x_n = t \}$
is $\lesssim V^{\frac{n-1}{n}}$.

The first key step in our proof is to understand how this slicing estimate interacts with the geometry of the nerve
$N$.  In order to do that, we intersect our manifold $X$ with slabs $Slab(j) := \{ j \le x_n \le j+1 \}$.

Let $X_j := X \cap Slab(j)$.  We view $X_j$ and $X$ as mod 2 chains, and observe that $X = \sum_j X_j$.  
(Since $X$ is a closed manifold, all but finitely many $X_j$ are empty - hence the sum $\sum_j X_j$ is
a finite sum.)

The manifold $X$ is closed, and so it is a mod 2 cycle.  But the chains $X_j$ may have boundary.  Since
we are working mod 2, $\partial X_j = X \cap \{ x_n = j \} + X \cap \{ x_n = j+1 \}$.  We define $Z_j$ to be
the mod 2 (k-1)-cycle $X \cap \{ x_n = j \}$.  So $\partial X_j = Z_j + Z_{j+1}$.  Each $Z_j$ is a cycle
in $X$.  And each $Z_j$ is null-homologous because it bounds $X \cap \{ x_n \ge j \}$.  

Similarly, we divide the nerve $N$ according to the slabs $Slab(j)$.  To begin with, we let $\frak B$ denote
our set of balls $\{ B(p_i, 1/4) \}$.  Next, we let $\frak B_j$ denote the set of balls in $\frak B$ which intersect
$Slab(j)$.  We define $N_j$ to be nerve corresponding to the set of balls $\frak B_j$.  In other words,
$N_j \subset N$, and a simplex of $N$ lies in $N_j$ if and only if all of the balls corresponding to the vertices
of the simplex intersect $Slab(j)$.

Our geometric estimate about the surface area of slices $N_1(X) \cap \{ x_n = t \}$ allows us to control
the number of simplices in each $N_j$.

\begin{lemma} For each $j$, the number of simplices in $N_j$ is $\lesssim V^{\frac{n-1}{n}}$.
\end{lemma}

\begin{proof} By a standard argument, we prove
that each point $x \in \mathbb{R}^n$ lies in $\lesssim 1$ balls of $\frak B$.  
Since the centers $p_i$ are $1/4$-separated, the balls $B(p_i, 1/8)$ are disjoint.
If a point $x$ lies in $B(p_i, 1/4)$, then $B(p_i, 1/8)$ lies in $B(x, 3/8)$.  There
are $\lesssim 1$ disjoint balls of radius $1/8$ in any ball of radius 3/8, and this
proves the estimate.

In particular, any point $x$ lies in $\lesssim 1$ balls of $\frak B_j$.  So to prove
the lemma, it suffices to prove that there are $\lesssim V^{\frac{n-1}{n}}$ balls
in $\frak B_j$.

To do this, we consider the doubles of the balls in $\frak B$: the set
$B(p_i, 1/2)$.  We let $U' := \cup_i B(p_i, 1/2) \subset N_1(X)$.  
If $B(p_i, 1/4)$ lies in $\frak B_j$, then $B(p_i, 1/2) \cap Slab(j)$ has
volume $\gtrsim 1$.  The balls $B(p_i, 1/2)$ also have multiplicity $\lesssim 1$,
and so the number of balls in $\frak B_j$ is bounded $\lesssim Vol_n [U' \cap Slab(j)]$.

This last volume is bounded above in terms of slices:

$$ Vol_n [U' \cap Slab(j) ] \le \int_{j}^{j+1} Vol_{n-1} [ N_1(X) \cap \{ x_n = t \} ] dt
\lesssim V^{\frac{n-1}{n}}. $$

This finishes the proof of the lemma. \end{proof}

Since the map $\Phi$ is subordinate to the cover of $U$ by the balls of $\frak B$,
$\Phi$ maps each $X_j$ into $N_j$.  Also, $\Phi$ maps $Z_j$ into $N_j$.

Next we deform $\Phi$ to a map $\Phi'$ which takes each $X_j$ into the
k-skeleton of $N_j$ and each $Z_j$ into the (k-1)-skeleton of $N_j$.
We define $X_j'$ to be the mod 2 k-chain given by $\Phi' (X_j)$ and we define
$Z_j'$ to be the mod 2 (k-1)-cycle given by $\Phi'(Z_j)$.  Because $Z_j$ was null-homologous
in $X$, $Z_j'$ is null-homologous in $N$.  Because of our bounds on
the geometry of $N$, we know that each $X_j'$ and each $Z_j'$ has 
$\lesssim V^{\frac{n-1}{n}}$ simplices.

Next we move our information about chains and cycles in the nerve $N$ back
to information about chains and cycles in $X$.

Recall that the nerve $N$ is homotopy equivalent to $U$.  We have a map
$\Psi: N \rightarrow U$ and $\Psi \circ \Phi$ is homotopic to the identity.  Since
$\Phi$ is homotopic to $\Phi'$, $\Psi \circ \Phi': U \rightarrow U$ is also homotopic
to the identity.

Now $U$ retracts to $X$.  Let $R: U \rightarrow X$ be the retraction.  So we see
that $R \circ \Psi \circ \Phi': X \rightarrow X$ is homotopic to the identity.  In particular,
$R \circ \Psi \circ \Phi' ( [ X ] )$ is homologous to $[X]$.  Therefore, we can decompose the
cycle $[X]$ in the following way:

$$ [ \sum_j R \circ \Psi( X_j ' ) ] = [X ].  \eqno{(1)}$$

To exploit our information about the hyperbolic metric on $X$, we use the simplex
straightening method developed by Thurston and Milnor in the 1970's \cite{MT}.

\newtheorem*{simpstraight}{Simplex straightening lemma}
\begin{simpstraight} (\cite{MT}) If $N$ denotes any simplicial complex and $X$ is a complete hyperbolic
manifold, and $F: N \rightarrow X$ is a continuous map, then $F$ may be homotoped to a
``straight map" $S: N \rightarrow X$ which maps each d-simplex of $N$ to a geodesic d-simplex
in $X$.  In particular, for $d \ge 2$ the image $S(\Delta^d)$ has $d$-volume $\lesssim 1$ for
each simplex $\Delta^d \subset N$.
\end{simpstraight}

In our case, we have a map $R \circ \Psi$ from $N$ to $X$, and we homotope it to a straight
map $S$.  Applying this homotopy to equation $(1)$, we get

$$[ \sum_j S (X_j') ] = [X]. \eqno{(2)}$$

We let $\bar X_j := S (X_j')$, and we let $\bar Z_j = S(Z_j')$.  The geometric properties of the straightening
map allow us to bound the volumes of $\bar X_j$ and $\bar Z_j$.  Because $X_j'$ consists of $\lesssim V^{\frac{n-1}{n}}$
simplices, the k-volume of $\bar X_j$ is $\lesssim V^{\frac{n-1}{n}}$.  Because $k \ge 3$ and $Z_j' $ consists of 
$\lesssim V^{\frac{n-1}{n}}$ simplices, the (k-1)-volume of $\bar Z_j$ is $\lesssim V^{\frac{n-1}{n}}$.

By equation (2), $\sum \bar X_j$ is homologous to $X$.  Since $\partial X_j' = Z_j' + Z_{j+1}'$, we see
that $\partial \bar X_j = \bar Z_j + \bar Z_{j+1}$.  Since $Z_j'$ are null-homologous in $N$, $\bar Z_j$ is
null-homologous in $X$.

At this point we can use the isoperimetric information about $X^k$.  Recall that $h$ is the Cheeger
isoperimetric constant of $X^k$ with its hyperbolic metric.  This means that if $U \subset X$
is an open set with volume $\le (1/2) V_{hyp}$, then the surface area of $\partial U$ is
$\ge h Vol(U)$.  Based on this definition, it follows that any null-homologous mod 2 (k-1)-cycle $Z$ in $X^k$
bounds a k-chain $Y$ with $Vol_k(Y) \le h^{-1} Vol_{k-1}(Z)$.  Since each of our cycles
$\bar Z_j$ is null-homologous, we can find a k-chain $\bar Y_j$ with $\partial \bar Y_j = \bar Z_j$ obeying the volume
estimate $ Vol_k (\bar Y_j) \lesssim h^{-1} Vol_{k-1} (\bar Z_j) \lesssim h^{-1} V^{\frac{n-1}{n}}$.

Now we can decompose the fundamental class of $X$ as a sum of cycles:

$$ [X] \sim \sum_j (\bar X_j + \bar Y_j + \bar Y_{j+1}).$$

When we sum over $j$, each $\bar Y_j$ appears twice and cancels.  Hence the right-hand side
is equal to $\sum_j \bar X_j$ which is indeed homologous to $[X]$.

On the other hand, each term $\bar X_j + \bar Y_j + \bar Y_{j+1}$ is a cycle because
its boundary is $\bar Z_j + \bar Z_{j+1} + \bar Z_j + \bar Z_{j+1}$ and we are working modulo 2.
So one of the cycles $(\bar X_j + \bar Y_j + \bar Y_{j+1})$ has volume at least
$V_{hyp}$.  Comparing this lower bound with our upper bounds for the volumes of $\bar X_j$
and $\bar Y_j$, we see

$$V_{hyp} \lesssim ( 1 + h^{-1}) V^{\frac{n-1}{n}}. $$

Finally, we note that any closed hyperbolic manifold has $h \lesssim 1$, and so the $h^{-1}$
dominates the 1.

$$V_{hyp} \lesssim h^{-1} V^{\frac{n-1}{n}}. $$
\end{proof}

{\bf Possible generalizations.} Theorem 3.2 can probably be generalized to many non-hyperbolic manifolds as well.  
We used hyperbolicity only to apply simplex-straightening and to bound the volumes of straight simplices
of dimension $k$ and $k-1$.  For example, if $X^k$ is a locally symmetric space with negative Ricci curvature, then
it is possible to apply simplex straightening, and the volume of a k-dimensional straight simplex was bounded by
Lafont and Schmidt in \cite{LS}.  To apply the proof above, one also needs a bound on the volumes of straight simplices
of dimension $k-1$.  If the universal cover of $X$ has a hyperbolic plane as a factor, then the volumes of straight simplices
of dimension $k-1$ will be unbounded.  But if the universal cover of $X$ does not have a hyperbolic plane as a factor,
then it may be possible to prove such a bound along the lines of \cite{LS}.

\subsection{Strong combinatorial thickness and retractions}

In this section, we give the proof of Theorem 3.1.  

\begin{theorem3.1} If $X^k$ is a k-dimensional simplicial complex with $N$ simplices and each vertex lies in
$\le L$ simplices, and if $n \ge 2 k + 1$, then there is an embedding with retraction thickness 1 from $X$ 
into $B^n(R)$ for $R \le C(n,L, \epsilon) N^{\frac{1}{n-k} + \epsilon}$.  
\end{theorem3.1}

We adapt the proof of Theorem 2.1.  In order to control retraction thickness, we need the following variation of combinatorial
thickness.

If $X$ is a k-dimensional simplicial complex, we say that a topological 
embedding $I: X \rightarrow \mathbb{R}^n$ has strong combinatorial
thickness $\ge T$ if the following holds.  Given any simplices $\Delta_1, ...
\Delta_J$ of $X$, then $\cap_{j=1}^J N_T[ I (\Delta_j) ]$ is non-empty only if
the simplices $\Delta_1, ..., \Delta_J$ have a common vertex.  (In other words,
if and only if $\cap_{j=1}^J \Delta_j$ is non-empty.)

For example, suppose $X$ is a triangle - a 1-complex with three edges, $E_1, E_2, E_3$.
If $I: X \rightarrow \mathbb{R}^3$ has combinatorial thickness at least 1, then it means
that each vertex is not mapped too close to the opposite edge.  If $I$ has strong combinatorial
thickness at least 1, then it means in addition that there is no point that lies too close to
all three edges.  The strong combinatorial thickness is related to retractions by the following
proposition.

\begin{prop} If $I: X \rightarrow \mathbb{R}^n$ is an embedding with strong combinatorial thickness
at least $T$, then the $T-\epsilon$ neighborhood of $I(X)$ retracts to $I(X)$.
\end{prop}

\begin{proof} 
In this proof, we will slightly abuse notation by identifying $I(X)$ and $X$.  When we
write $N_T(X)$, this means the $T$-neighborhood of $X$ in $\mathbb{R}^n$.

Pick a fine triangulation of $N_{T- \epsilon} (X)$, where each edge of the
triangulation has length $< \epsilon$.  We first construct a map $\Psi$ from $N_{T - \epsilon}(X)$
to $X$, building it one skeleton at a time.  Next we will check that the restriction of $\Psi$ to 
$X$ is homotopic to the identity.  In other words, $\Psi$ is a homotopy retraction of $N_{T - \epsilon}
(X)$ onto $X$.  But since $I$ is an embedding, the homotopy extension theorem allows us
to homotope $\Psi$ to a genuine retraction from $N_{T - \epsilon} (X)$ onto $X$.

We have to be careful with notation because $X$ is a simplicial complex, and we need to talk
about its simplices, and we also triangulated $N_{T - \epsilon}(X)$ and we need to think about
the simplices of the triangulation.  We use $\Delta^j$ to refer to a j-dimensional simplex of 
$X$.  We use English letters to refer to simplices of the triangulation of $N_{T - \epsilon}(X)$:
we call the vertices $v$ and the higher-dimensional faces $F$.

We first define $\Psi$ on the vertices of our triangulation of $N_{T- \epsilon} (X)$.  
For any vertex $v$, consider all the simplices $\Delta \subset X$ so that $v \in N_T [ \Delta ] $.
By the definition of strong combinatorial thickness, {\it all of these simplices have a common
vertex}.  We define $\Psi (v)$ to be one of the common vertices.

Now we claim that $\Psi$ extends to a simplicial map from $N_{T- \epsilon} (X)$ to $X$.
Let $F$ denote a d-simplex of $N_{T- \epsilon} (X)$ with vertices $v_1, ..., v_{d+1}$.  
We have to check that $\Psi(v_1), ..., \Psi(v_{d+1})$ are all vertices of a single simplex of $X$.
By hypothesis, $v_1 \in N_{T- \epsilon} (X)$, and so there is a simplex $\Delta \subset X$ 
with $v_1 \in N_{T- \epsilon} (\Delta)$.  Since the edges of the triangulation of $N_{T- \epsilon} (X)$
have length $< \epsilon$, all the vertices $v_i$ lie in $N_{T} (\Delta)$.  Now by our choice of $\Psi(v_i)$, we
see that each $\Psi(v_i)$ lies in $\Delta$.  Hence $\Psi$ extends to a simplicial map.

Next we have to check that $\Psi$ is homotopic to the identity.  To do this, we check that for
each simplex $\Delta^j \subset X$, $\Psi(\Delta^j)$ lies in the closure of $\Delta^j$.   Let $p$ be a point
of some simplex $\Delta^j \subset X \subset \mathbb{R}^n$.  The point $p$ lies in some
simplex $F$ of our fine triangulation of $N_{T- \epsilon} (X)$.  We let $v_1, ..., v_{d+1}$
denote the vertices of $F$.  Because we chose a fine triangulation, we can assume that
the distance from $p$ to any $v_i$ is $\le \epsilon$.  In particular, each $v_i$ easily lies
in $N_T (\Delta^j)$.  Therefore $\Psi(v_i)$ is one of the vertices of $\Delta$, and $\Psi(F)$ 
lies in the closure of $\Delta^j$.  In particular $\Psi(p)$ lies
in the closure of $\Delta^j$.  Since $p$ was an arbitrary point of $\Delta^j$, $\Psi(\Delta^j)$ lies in
the closure of $\Delta^j$.  

Therefore, $\Psi$ restricted to $X$ is homotopic to the identity.  By the homotopy extension property,
$\Psi$ is is homotopic to a retraction.  \end{proof}

If an embedding $I: X \rightarrow \mathbb{R}^n$ has strong combinatorial thickness
$ > T$, then it has retraction thickness $> T$ also.  Therefore, Theorem 3.1 follows
from the following proposition.

\begin{prop} Suppose that $X$ is a k-dimensional simplicial complex with $N$ simplices
and that each vertex of $X$ lies in $\le L$ simplices.  If $n \ge 2k+1$, then there is an
embedding $I: X \rightarrow B^n(R)$ with strong combinatorial thickness $> 1$, where
$R \le C(n,L) N^{\frac{1}{n-k}} (polylog N)$.  In particular, the 1-neighborhood of $I(X)$
retracts to $I(X)$.
\end{prop}

This proposition is a modification of Theorem 2.1.  We control the strong combinatorial thickness
instead of the combinatorial thickness.  On the other hand, the estimate is a little worse, with
a $polylog N$ error factor in place of $(\log N)^{2k+2}$.  If one follows the proof, the power in
the $polylog N$ turns out to be $\sim (\log N)^{C k^4}$.  The proof occupies the rest of
Section 3.4.

This proof is a modification of the proof of Theorem 2.1 in Section 2.1.
The proof begins in the same way.  We let $I_0$ be a random facewise-linear embedding
of $X$ into $B^n(N^{\frac{1}{n-k}})$.  We know that each unit ball $B^n(p,1)$ intersects
$\lesssim \log N$ simplices of $I_0(X)$.  Next we subdivide $X$ into a finer complex $X'$,
so that each simplex of $X'$ has diameter $\lesssim 1$.  Then we perturb the map $I_0$ to
a map $I: X' \rightarrow B^n$ by moving each vertex a distance $\le 1$.  After the perturbation,
we need to check that the strong combinatorial thickness of $I: X' \rightarrow \mathbb{R}^n$ 
is $\ge \epsilon$ for some $\epsilon \gtrsim (polylog N)^{-1}$.

For any set of simplices $\Delta_1, ..., \Delta_J \subset X'$, we say that $\Delta_1, ... \Delta_J$ have
a common vertex if one vertex lies in all of the simplices.  For any set of simplices $\Delta_1, ..., 
\Delta_J$ with no common vertex,
we define $Bad_\epsilon(\Delta_1, ..., \Delta_J)$ to be the set of choices of $I$ so
that $\cap_{j=1}^J N_\epsilon[ I(\Delta_j) ]$ is non-empty.  
The map $I$ has strong combinatorial thickness at least $\epsilon$ unless it lies in one
of the sets $Bad_\epsilon$.

For the original combinatorial thickness, the bad sets only involve two simplices.  The 
new difficulty of strong combinatorial thickness is that the bad sets involve more
simplices.  However, we can put some bound on the number of simplices involved.

\begin{lemma} If the strong combinatorial thickness of $I$ is $< \epsilon$, then
we can find some $J \le k+2$ and simplices $\Delta_1, ..., 
\Delta_J \subset X'$ with no common vertex, so that the map $I$ lies in $Bad_\epsilon(\Delta_1, ..., \Delta_J)$.
\end{lemma}

\begin{proof} Since $I$ has combinatorial thickness $< \epsilon$ it lies in some
set $Bad_\epsilon(\Delta_1, ..., \Delta_J)$, where a priori $J$ may be too large.
We have to find a subset of (k+2) of the simplices $\Delta_1, ..., \Delta_J$ with
no common vertex.  We include $\Delta_1$ in the subset.  We let $v_1$, ...,
$v_{k+1}$ be the vertices of $\Delta_1$.  
For each vertex $v_i$, we can find a simplex that does not contain $v_i$, because no
vertex lies in all of the simplices.  In this way, we choose $\le k+1$ more simplices,
and no vertex lies in all of them. \end{proof}

From now on, we only consider bad sets with at most $k+2$ simplices.

As in the proof of Theorem 2.1, most of the sets $Bad_\epsilon(\Delta_1, ..., \Delta_J)$
are empty.  They are non-empty only if all of the sets $I_0(\Delta_i)$ intersect a single
ball of radius $(1 + \epsilon) < 2$.  From now on, we consider only the non-empty 
sets.  We say that a vertex $v \in X'$ is involved in a bad set $Bad_\epsilon(\Delta_1, ..., \Delta_J)$
if $v$ is a vertex of one of the simplices $\Delta_1, ..., \Delta_J$.  
A ball of radius 2 intersects $\lesssim \log N$ simplices of $X'$.  Therefore,
each vertex $v_i$ is involved in $\lesssim (\log N)^{k+1}$ bad sets $Bad_\epsilon(\Delta_1, ..., \Delta_J)$.  
On the other hand, the number of vertices involved in a single bad set $Bad_\epsilon(\Delta_1, ..., \Delta_J)$ is
$\le (k+1)J \le (k+1)(k+2)$.  The last main ingredient is a bound for the probability of a bad set: the probability of
any set $Bad_\epsilon(\Delta_1, ..., \Delta_J)$ is $\lesssim \epsilon^{\alpha}$ for some power
$\alpha > 0$.

\begin{lemma} Let $J \le k+2$.  Then the probability of any bad set
$Bad_\epsilon(\Delta_1, ..., \Delta_J)$ is $\lesssim \epsilon^{\frac{1}{k+1}}$.
\end{lemma}

This inequality is probably not sharp.  I believe that the probability of $Bad_\epsilon(\Delta_1, ..., \Delta_J)$
is $\lesssim \epsilon$, but I don't know how to prove it.  This lemma is good enough for our application.

The rest of the proof of Proposition 3.7 is the same as the proof of Theorem 2.1.  We summarize it 
here.  We choose $v_1$, $v_2$, ....  As we make each choice, we keep track of the conditional probability of the
sets $Bad_\epsilon(\Delta_1, ..., \Delta_J)$.  When we pick $v_i$, the conditional
probability of $Bad_\epsilon(\Delta_1, ..., \Delta_J)$ changes only if $v_i$ is a vertex of $\Delta_1$, ..., or $\Delta_J$.  So the
choice of $v_i$ affects $\lesssim (\log N)^{k+1}$ conditional probabilities.  Therefore, we can choose $v_i$ so that the conditional probabilities of $Bad_\epsilon(\Delta_1, ..., \Delta_J)$
increase by at most a factor $\lesssim (\log N)^{k+1}$.   Each bad set $Bad_\epsilon(\Delta_1, ..., \Delta_J)$ involves
$\le (k+1)J \le  k^2 + 3k$ vertices.   Therefore, after we have chosen all $v_i$, the conditional
probability of $Bad_\epsilon(\Delta_1, ..., \Delta_J)$ has increased by a factor $\lesssim (\log N)^{(k+1)(k^2 + 3k)}$.  Since the probability
of $Bad_\epsilon(\Delta_1, ..., \Delta_J)$ was $ \lesssim \epsilon^\frac{1}{k+1}$ before making any choices, after all our choices, the conditional
probability of $Bad_\epsilon(\Delta_1, \Delta_2)$ is $\lesssim \epsilon^{\frac{1}{k+1}} (\log N)^{Ck^3}$.  Of course, once we have chosen all $I(v_i)$, the conditional probability of $Bad_\epsilon(\Delta_1, ..., \Delta_J)$ is either 1 or 0, depending on whether our choice lies
in $Bad_\epsilon(\Delta_1, ..., \Delta_J)$ or not.  We pick $\epsilon$ sufficiently small so that all conditional probabilities
of bad sets are strictly less than 1.  This $\epsilon$ is $\sim (\log N)^{-Ck^4}$.  Now our choice
of the map $I$ does not lie in any of the sets $Bad_\epsilon(\Delta_1, ..., \Delta_J)$, and $I$ has strong combinatorial thickness
$\ge \epsilon \gtrsim (\log N)^{- C k^4}$.  

It remains only to prove Lemma 3.9.  This proof consists of a quantitative analysis of the standard
general position argument for linear maps.  Quantitative analysis of general position
arguments for random linear maps is an important, well-studied area in numerical analysis, see \cite{De}
and \cite{BCL}.  We use here a fundamental result from this area, which estimates the probability
that the determinant of a random matrix is small.

\begin{prop} Let $M$ be a $d \times d$ matrix.  Consider $M$ as a point in $\mathbb{R}^{d^2}$.
Suppose that $M$ is chosen uniformly at random from the unit ball in $\mathbb{R}^{d^2}$.
Then for any $\delta > 0$, with probability $(1-\delta)$,
$| Det(M) | \ge c(d) \delta$.   Hence, with probability $(1 - \delta)$, the norm of
$M^{-1}$ is $\le c'(d) \delta^{-1}$. 
\end{prop}

I don't know who first proved this estimate.  It appears in \cite{De}.  There is a thorough proof in
\cite{BCL}.   In these references, they give very precise information about the asymptotics of
$c(d)$ and $c'(d)$, which is important for numerical linear
algebra and other applications.  The following generalization of this estimate is an immediate corollary.

\begin{prop} Let $\mu$ be a probability measure on the space of $d \times d$ matrics
$\mathbb{R}^{d^2}$.  Suppose that $\mu$ is supported on a ball of radius $\lesssim 1$
and that the density function of $\mu$ is pointwise $\lesssim 1$.  Then for any $\delta > 0$, with probability
$(1 - \delta)$, $| Det(M) | \ge c(d) \delta$.   Hence, with probability $(1 - \delta)$, the norm of
$M^{-1}$ is $\le c'(d) \delta^{-1}$. 
\end{prop}

Using Proposition 3.11, we can bound the chance that two faces of $I(X')$ are too close
together.  (Recall that $I$ is a random perturbation of size 1 from the initial embedding $I_0$,
and we want to measure the probability that two faces are very close together after
the random perturbation.)  In the following, we use $\Delta_i$ to refer to simplices of $I(X')$.  We write
$P_i$ to refer to the plane spanned by $\Delta_i$.

\begin{prop} If $\Delta_1$ and $\Delta_2$ are simplices of $I(X')$ with no vertices in common,
then $N_{\epsilon_1}(\Delta_1)$ and $N_{\epsilon_2}(\Delta_2)$ are disjoint with probability
$1-\delta$ for $\delta \lesssim (\epsilon_1 + \epsilon_2)$.  Moreover, $N_{\epsilon_1}(P_1)$ 
and $N_{\epsilon_2}(P_2)$ are disjoint with probability
$1-\delta$ for $\delta \lesssim (\epsilon_1 + \epsilon_2)$.
\end{prop}

\begin{proof} Let $v_0, ..., v_j$ be the vertices of $\Delta_1$ and $v_{j+1}, ..., v_l$ be the vertices
of $\Delta_2$.  Since $\Delta_1$ and $\Delta_2$ have dimension $\le k$, the number $l$ is at most
$2k +1 \le n$.  We translate our coodinates so that $v_0$ is at zero.  
If $N_{\epsilon_1}(P_1)$ and $N_{\epsilon_2}(P_2)$ intersect, then
we can choose affine combinations $c_0 v_0 + ... + c_j v_j$ and $c_{j+1} v_{j+1} + ...
+ c_l v_l$ which lie within $\epsilon_1 + \epsilon_2$ of each other.  In other words,
$| c_0 \cdot 0 + c_1 v_1 + ... + c_j v_j - c_{j+1} v_{j+1} - ... - c_l v_l | \le \epsilon_1 + \epsilon_2$.
We let $\vec{c}$ denote the l-dimensional column vector $(c_1, ..., c_j, - c_{j+1}, ..., - c_l)$.
We let $M$ denote the $l \times l$ matrix whose $i^{th}$ column is given by the first $l$ components
of $v_i$.  Then $|M \vec{c}| \le \epsilon_1 + \epsilon_2$.  Because $\sum_{j+1}^l c_i = 1$,
$|\vec{c}| \sim 1$.  Hence we get $| M^{-1} | \gtrsim (\epsilon_1 + \epsilon_2)^{-1}$.

The matrix $M$ depends on the coordinates of the vertices $v_i$ of $\Delta_1$ and $\Delta_2$.  The vertices are
chosen uniformly at random in unit balls.  Therefore, $M$ is chosen by a probability distribution which
obeys the hypotheses of Proposition 3.11.  Hence the probability that $| M^{-1} | \gtrsim (\epsilon_1 + \epsilon_2)^{-1}$
is $\lesssim (\epsilon_1 + \epsilon_2)$.  \end{proof}

Similarly, if $\Delta_1$ and $\Delta_2$ have vertices in common, then with high probability,
$N_{\epsilon_1}(P_1) \cap N_{\epsilon_2}(P_2)$ lies near to $P_1 \cap P_2$.

\begin{prop} For any $\delta > 0$, with probability $(1-\delta)$, $N_{\epsilon_1}(P_1) \cap
N_{\epsilon_2}(P_2)$ lies in $N_{\epsilon}(P_1 \cap P_2)$ for $\epsilon \sim \delta^{-1}
(\epsilon_1 + \epsilon_2)$.
\end{prop}

\begin{proof} Suppose that there is a point $x$ in $N_{\epsilon_1}(P_1) \cap
N_{\epsilon_2}(P_2)$ but not in the $\epsilon$-neighborhood of $P_1 \cap P_2$.

Let $w_0, ..., w_j$ be the vertices of $\Delta_1 \cap \Delta_2$.  The plane spanned by
$w_0, ..., w_j$ is $P_1 \cap P_2$ (almost surely).  Let
$v_1, ..., v_l$ be the other vertices of $\Delta_1$.  Let $v_{l+1}, ..., v_m$ be the
other vertices of $\Delta_2$.  We project to the orthogonal complement of
$P_1 \cap P_2$.  In this orthogonal complement, we make the
image of $P_1 \cap P_2$ the origin.  We denote the projection of $v_i$ to the orthogonal complement by
$\bar v_i$.  The projection of $\Delta_1$ is a simplex
spanned by $0, \bar v_1, ... \bar v_l$.  The projection of $\Delta_2$ is a simplex
spanned by $0, \bar v_{l+1}, ..., \bar v_m$.  The complement has dimension
$n  - j$, and because of the dimension assumptions on $\Delta_1, \Delta_2$, we see
that $m \le n-j$.

We know that $x$ lies in $N_{\epsilon_1}(P_1) \cap N_{\epsilon_2}(P_2)$.
The projection of $x$ lies within $\epsilon_1$ of some vector in the projection of
$P_1$, which has the form of an affine combination
$c_0 \cdot 0 + c_1 \bar v_1 + ... + c_l \bar v_l$.  Since $x$ does not lie in
$N_\epsilon(P_1 \cap P_2)$ the projection of
$x$ has norm $> \epsilon$.  Therefore, $| \sum_{i=1}^l c_i  | \gtrsim \epsilon$.

The projection of $x$ also lies within
$\epsilon_2$ of some vector in the projection of $P_2$, which has the form of 
an affine combination $d_0 \cdot 0 + d_1 \bar v_{l+1} + ... + d_{m-l} \bar v_m$.
Therefore, these two points are close to each other:

$$| c_1 \bar v_1 + ... + c_l \bar v_l - d_1 \bar v_{l+1} - ... - d_{m-l} \bar v_m | \le \epsilon_1 + 
\epsilon_2. \eqno{(*)}$$

We let $\vec{c}$ be the vector $(c_1, ..., c_l, - d_1, ... , - d_{m-l})$.  It has $m$ entries.
We let $M$ be the $m \times m$ matrix with columns given by the first $m$
entries of $\bar v_i$.  Inequality $(*)$ implies that $| M \vec{c} | \le \epsilon_1 + \epsilon_2$.
Since $| \sum_{i=1}^l c_i  | > \epsilon$, we see that $| \vec{c} | \gtrsim \epsilon$ and so
$|M^{-1}| \gtrsim \epsilon (\epsilon_1 + \epsilon_2)^{-1}$.

Now the entries of $M$ depends on the randomly chosen positions of the vertices of $\Delta_1$
and $\Delta_2$, and it's not hard to check that $M$ is a random matrix obeying the conditions of
Proposition 3.11.  Therefore, the probability that $|M^{-1}| \gtrsim \epsilon (\epsilon_1 + \epsilon_2)^{-1}$
holds is $\lesssim \epsilon^{-1} (\epsilon_1 + \epsilon_2)$. \end{proof}

With these tools, we are ready for the proof of Lemma 3.9.

\begin{proof}[Proof of Lemma 3.9]

We have a sequence of simplices of $I(X')$, $\Delta_1, \Delta_2, ..., \Delta_J$, with $J \le k+2$.  We 
know that $\cap_{j=1}^J \Delta_j$ is empty, and we can assume that $\cap_{i=1}^{J-1} \Delta_i$ is not
empty.  We inductively analyze the intersection $\cap_{i=1}^j N_\epsilon (P_i)$.  

{\bf Inductive claim.} For any $\delta > 0$, for each $1 \le j \le J-1$, the following estimate holds with probability $( 1 - j \delta)$:

$$ \cap_{i=1}^j N_\epsilon(P_i) \subset N_{C \delta^{1-j} \epsilon}(\cap_{i=1}^j P_i). $$

{\bf Proof of claim.} The case $j=1$ is trivial.  The step from $j-1$ to $j$ follows from Proposition 3.13.  Note that
$\cap_{i=1}^j P_i$ is the plane spanned by $\cap_{i=1}^j \Delta_i$.

Now $\cap_{j=1}^{J-1} \Delta_i$ and $\Delta_J$ are disjoint.  So Proposition 3.12 tells us that
$N_{C \delta^{2-J} \epsilon}((\cap_{i=1}^{J-1} P_i)$ and $N_\epsilon(P_J)$ intersect
with probability $\lesssim \delta^{2 -J} \epsilon$.  

Hence $\cap_{j=1}^J N_\epsilon(\Delta_j)$ is non-empty with probability $\lesssim \delta + \delta^{2-J} \epsilon \le \delta + \delta^{-k}\epsilon$.  At this point we choose $\delta$ to optimize our bound: $\delta = \epsilon^{\frac{1}{k+1}}$.

\end{proof}

\subsection{Open questions}

Let us say that the retraction-thick embedding radius of $X^k$ into $\mathbb{R}^n$ is the
smallest $R$ so that $X^k$ embeds with retraction thickness 1 into $B^n(R)$.

Suppose that $X^k$ is a simplicial complex with $N$ simplices and with each vertex
lying in $\lesssim 1$ simplices.  How big can the retraction-thick embedding radius of
$X$ be?

For $k=1$, the retraction thick embedding radius of $X$ must be $\lesssim N^{1/n}$, and this
is sharp.  For $k \ge 2$, we don't know what happens.

By Theorem 3.1, the retraction-thick embedding radius of $X^k$ into $\mathbb{R}^n$ is
$\lesssim N^{\frac{1}{n-k} + \epsilon}$.

For example, the retraction-thick embedding radius of $X^2$ into $\mathbb{R}^n$ is
$\lesssim N^{\frac{1}{n-2} + \epsilon}$.  But in the examples that we could work out, the
retraction thick embedding radius was at most $N^\frac{1}{n}$.  Which 2-complexes are
hardest to embed in Euclidean space?

The most complicated complexes that we considered were the arithmetic hyperbolic
manifolds of dimension $k \ge 3$.  The retraction-thick embedding radius of such $X$ into
$\mathbb{R}^n$ is $\gtrsim N^{\frac{1}{n-1}}$.  The most difficult complexes are probably 
not manifolds.  What are they, and how can we exploit them?  For example, one may want to
consider k-skeleta of higher-dimensional arithmetic manifolds or Garland buildings as in \cite{Gr2}.

On the other hand, one may ask how big the retraction-thick embedding radius can be for 
a manifold $X^k$ embedded in $\mathbb{R}^n$, if $X$ can be triangulated with $N$ simplices.  
For $k = 1, 2$, the retraction-thick embedding
radius is $\lesssim N^{1/n}$, and this is sharp.  For $k = 3$, the situation is unclear.  The retraction-thick
embedding radius is $\lesssim N^{\frac{1}{n-3}}$.  For arithmetic hyperbolic 3-manifolds, it
is $\gtrsim N^{\frac{1}{n-1}}$.  Are arithmetic hyperbolic 3-manifolds the hardest 3-manifolds to embed
in Euclidean space?  What is their actual retraction-thick embedding radius?

\section{Distortion of knots}

Given a set $K \subset \mathbb{R}^n$, we write $dist_K$ to denote the intrinsic distance in $K$: 
$dist_K(y,z)$ denotes the infimal length of a path in $K$ from $y$ to $z$.  The distortion of $K$ measures
the ratio between the intrinsic distance and the extrinsic distance:

$$distor(K) := \sup_{x,y \in K} \frac{dist_K (x,y) }{dist_{\mathbb{R}^3} (x,y)}. $$

In this section, we discuss the distortion of knots $K \subset \mathbb{R}^3$.  The distortion is a scale-invariant
measure of the geometric complexity of a knot $K$.  We want to understand how the distortion of a knot is
related to the topology of the knot.  One fundamental question is whether there are isotopy classes of knots that
require arbitrarily large distortion.  This question was raised in the early 1980's (\cite{Gr3}), and it was open for
many years.  Recently the question was resolved by Pardon.

\newtheorem*{theorem4.1}{Theorem 4.1}
\begin{theorem4.1} (Pardon, \cite{P}) There are isotopy classes of knots in $\mathbb{R}^3$ requiring arbitrarily large distortions.
\end{theorem4.1}

In this section we give an alternate proof of Theorem 4.1.

Let's say a little about why the problem is difficult.  First of all, there are infinitely
many isotopy classes of knots with distortion $< 100$.  For example, one may consider a long chain of simple knots.  Such 
a long chain is reducible, but there are also infinitely many irreducible isotopy classes.  For example, take a long chain of simple
knots, and then loop each knot in the chain through the adjacent knots.  See \cite{O} for details.  A knot of distortion $< 100$ may
also be knotted at arbitrarily many different scales: take a trefoil knot, and then zoom in on a small nearly straight curve and replace it
by a small trefoil knot, and then zoom in on it, etc.  No matter how many times we repeat such an operation, the distortion will
be $< 100$.  Also, one may replace trefoils by other simple knots: figure eight knots, etc.  Let $K(D)$ denote the set of isotopy classes
of knots with distortion $< D$.  One might like to find a knot invariant which is uniformly bounded on $K(100)$, and then to find
another knot $K'$ where the invariant is outside these bounds.  Then we could conclude that $K'$ cannot be isotoped to have
distortion $< 100$.  But $K(100)$ contains infinitely many knots, and it appears that all standard invariants are unbounded on
$K(100)$.  (For example, one may consider the genus of the knot, the unknotting number, the degree of the Alexander polynomial, the size of the coefficients of the Alexander polynomial, the degree or size of other knot polynomials, the size of some Vassiliev invariant, etc.
I suspect that these are all unbounded on $K(100)$. )

On the other hand, there are many isotopy classes of knots that appear to have large distortions.  For example, the torus
knots $T_{p,q}$ appear to have large distortion if $p$ or $q$ is large.  Pardon gave the following bound for distortion of torus
knots.

\newtheorem*{torusknotdis}{Distortion of torus knots}
\begin{torusknotdis} (\cite{P}) If $K$ is a $(p,q)$-torus knot with $2 \le p < q$,
then the distortion of $K$ is $\gtrsim p$.
\end{torusknotdis}

This lower bound is sharp in some cases.  For example, the distortion of $T_{p, p+1}$ is $\lesssim p$.
But in other cases it looks far from sharp.  For example, it appears that the distortion of $T_{2, q}$ is
$\sim q$.   (The results in \cite{P} also apply to knots on higher-genus surfaces.)

Pardon's argument and our argument both involve another measure of the geometric complexity of a knot,
called the conformal length.  Define the conformal length of a knot $K \subset \mathbb{R}^3$ as

$$ \sup_{x \in \mathbb{R}^3, r > 0} r^{-1} length [ K \cap B(x,r) ] . $$

\noindent Similarly, we can define the conformal k-volume for any k-dimensional polyhedron in $\mathbb{R}^n$.
We denote conformal k-volume by $convol_k$ and in particular conformal length by $convol_1$.  
The conformal volume is a small variation of a conformal invariant defined by Li and Yau in \cite{LY}.

The conformal length is another scale-invariant measure of the geometric complexity of a knot $K$.
It is related to the distortion by the following lemma, which essentially appears in \cite{P}.

\begin{lemma} If $K \subset \mathbb{R}^3$ is a knot, then

$$ convol_1(K) \le 4 distor(K) .$$

\end{lemma}

\begin{proof} Pick any $x \in \mathbb{R}^3$ and $r > 0$.  Let $L$ denote the length of $K \cap
B(x,r)$.  We have to prove that $L/r \le 4 distor(K)$.

Pick two points $z, y$ in $K \cap B(x,r)$ so that $dist_K (y,z)$ is at least $L/2$.  On the other hand,
$dist_{\mathbb{R}^3} (y,z) \le 2 r$.

$$ \frac{L}{4r} \le \frac{dist_{K} (y,z) }{dist_{\mathbb{R}^3} (y,z)} \le distor(K) . $$

\end{proof}

Therefore, proving that a knot has large conformal length is even stronger than proving that it has
large distortion.  In \cite{P}, Pardon estimated the conformal length of torus knots.

\newtheorem*{torusknot}{Conformal lengths of torus knots}
\begin{torusknot} (\cite{P}) If $K$ is a $(p,q)$-torus knot with $2 \le p < q$,
then the conformal length of $K$ is $\gtrsim p$.
\end{torusknot}

This estimate is sharp up to constant factors for every torus knot: the knot $T_{p,q}$ can
be realized with conformal length $\lesssim p$.  The knot $T_{p,q}$ can be thought of as a braid
with $p$ strands, where the two ends of the braid are connected up in a simple way.  Any
$p$-strand knot can be arranged to have conformal length $\lesssim p$ by making the strands very 
long and having them move around each other slowly.

We will prove lower bounds for conformal lengths of knots associated to arithmetic hyperbolic
3-manifolds.  To define these knots, we recall a theorem of Montesinos and Hilden.

\begin{hmthm} (\cite{H}, \cite{M}) Any closed oriented 3-manifold $M$ admits a degree 3 map
$M \rightarrow S^3$ which is a ramified cover ramified over a knot.
\end{hmthm}

When $M$ is a complicated hyperbolic manifold, and $M$ has a degree 3 map to $S^3$ ramified
over a knot $K$, then we prove that $K$ has a large conformal length.

\newtheorem*{theorem4.2}{Theorem 4.2}
\begin{theorem4.2} If $M$ is a closed oriented hyperbolic 3-manifold with volume $V$ 
and Cheeger constant $h$, and if $F: M \rightarrow S^3$ is a 3-fold cover ramified over
a knot $K$, then the conformal length of $K$ is $\gtrsim h V$.
\end{theorem4.2}

We will measure the conformal length of $K$ in $\mathbb{R}^3$.  We pick any point $q \in
S^3 \setminus K$, and we identify $S^3 \setminus q$ with $\mathbb{R}^3$.  The resulting
knot $K$ in $\mathbb{R}^3$ always has conformal length $\gtrsim h V$.  (With a similar
argument, one can also bound the conformal length of $K$ in $S^3$ with the round metric.)

The conclusion also applies to any knot $K'$ isotopic to $K$.  Suppose that
$\Psi: S^3 \rightarrow S^3$ is a diffeomorphism taking $K$ to $K'$.  Then $\Psi \circ F: M \rightarrow
S^3$ is a degree 3 cover ramified over $K'$ - and so the conformal length of $K'$ is also
$\gtrsim h V$.  

The most interesting examples occur when $M$ is an arithmetic hyperbolic 3-manifold.
See the appendix for more background on arithmetic hyperbolic 3-manifolds.  
Suppose that we take a fixed closed arithmetic hyperbolic 3-manifold $M_0$ and then
look at a sequence of arithmetic covers $M_i$ with volume $V_i \rightarrow \infty$.  Remarkably,
the Cheeger constants of these $M_i$ are bounded below uniformly: $h_i \ge c > 0$.
Applying the Hilden-Montesinos theorem, we can map $M_i \rightarrow S^3$ degree 3
ramified over a knot $K_i$.  By Theorem 4.2, the conformal length of $K_i$
is $\gtrsim V_i \rightarrow + \infty$.  In particular, the distortion of $K_i$ is $\gtrsim V_i
\rightarrow + \infty$.  So Theorem 4.2 implies Theorem 4.1.

Our lower bound for the conformal length is essentially sharp in this example.  If $M_i$
is a degree $D_i$ cover of $M_0$, then $D_i \sim V_i$.  If $G_0$ denotes the Heegaard 
genus of $M_0$, then the Heegaard genus of $M_i$ is $G_i \le D_i 
G_0 \lesssim V_i$.  Hilden's construction of the degree 3 cover $M_i \rightarrow S^3$ ramified
over $K_i$ is based on a Heegaard decomposition of $M_i$.  The construction of $K_i$ gives a braid
with $\lesssim G_i$ strands, capped at each end in a simple way.  Now we can isotope $K_i$ so that
the braid is very long and the strands move around each other very slowly.  For such a knot $K_i$,
the conformal length is $\lesssim G_i \lesssim V_i$.  On the other hand, it's not clear how sharp
our lower bound for the distortion of $K_i$ is.

Here is an outline of our proof.  One way to understand the complexity of a knot $K$ is to
build a triangulation of $S^3$ so that the knot is an embedded curve in the 1-skeleton.  If
the triangulation is not too complicated, then the knot cannot be too complicated either.  We use
our information about the conformal length of $K$ to build such a triangulation which is ``not 
too big".  Now, the simplest way to measure the size of a triangulation is the number of simplices.
But we cannot bound the number of simplices in the triangulation in terms of the conformal length
of $K$, or even in terms of the distortion of $K$.  This is just because there are infinitely many
different knots with distortion $< 100$.  Instead, we bound a kind of ``width" of the triangulation.

More precisely, we build a simplicial map $G$ from our triangulated 3-sphere $(S^3, Tri_0)$ 
to a tree $T$ so that every fiber of the
map is small.  The fiber over a vertex of $T$ is a 2-dimensional subcomplex of $(S^3, Tri_0)$ containing
$\lesssim convol_1(K)$ simplices, and the fiber over an edge of $T$ also contains $\lesssim
convol_1(K)$ simplices.

The main difficulty in building this triangulation is just to find a map from $S^3$ to a tree $T$ whose
fibers have controlled topology and don't intersect $K$ too many times.  In particular, we find a map
$S^3 \rightarrow T$ where every fiber intersects $K$ in $\lesssim convol_1(K)$ points and where
each fiber is homotopic to a 2-sphere or to a bouquet of 2-spheres with $\lesssim 1$ spheres.

To get an idea how the conformal length may be used to construct such a map, consider a cube 
$Q \subset \mathbb{R}^3$ with
side-length $s$.  We know that the length of $K \cap 2Q$ is $\lesssim convol_1(K) s$.
Now if we translate $Q$ by a random vector of length $\le s$, the average
number of intersections between $K$ and $\partial Q$ is $\lesssim convol_1(K)$.  The fibers of the map $F$ will be 
the boundaries of this kind of cube (or the union of a few such boundaries).  

In the second half of the proof, we connect this nice triangulation with the hyperbolic geometry of $M$.
Since $F: M \rightarrow S^3$ is a degree 3 cover ramified over $K$, and since $K$ is embedded in
the 1-skeleton of our triangulation $Tri_0$, we can lift $Tri_0$ to a triangulation $Tri$ of $M$.  Composing $F$
and $G$, we get a simplicial map from $(M, Tri)$ to the tree $T$ with small fibers.  
We can exploit the hyperbolic geometry of $M$ by straightening the simplices of $Tri$.  The simplex straightening
gives a map $S: (M, Tri) \rightarrow (M, hyp)$, homotopic to the identity, mapping each 2-simplex of $Tri$ to a geodesic 
2-simplex of $(M,hyp)$ with area $\lesssim 1$ and each 3-simplex of $Tri$ to a geodesic 3-simplex of $(M^3, hyp)$
with volume $\lesssim 1$.

Now in rough terms, the straightened version of $(M, Tri)$ is ``thin" - it has ``width" $\lesssim convol_1(K)$.  On
the other hand, the hyperbolic manifold $(M, hyp)$ is ``wide" - it has width $\gtrsim h V$.  Since the straightening
map $S$ has degree 1, $(M, Tri)$ must be wide enough to fit around $(M, hyp)$ and so $convol_1(K) \gtrsim h V$.

The last step of this argument is similar to an estimate from \cite{Gr1}, which says that a sufficiently generic map
from $M^3$ to a tree $T$ has a fiber with genus $\gtrsim h V$.  We include a self-contained proof which is
adapted to our situation, but the main idea comes from \cite{Gr1}.

We remark that our lower bound for the distortion of the knot $K$ is ultimately powered by the expander
properties of arithmetic hyperbolic 3-manifolds.  So to some extent, this theorem may be considered
a generalization of the Kolmogorov-Barzdim estimate on the difficulty of embedding an expander into
$\mathbb{R}^3$.

Now we turn to the detailed proof of the theorem.

\begin{proof}[Proof of Theorem 4.2]

It is convenient to assume that the knot $K$ is piecewise linear.  If the knot is originally smooth or piecewise
smooth, we can approximate it by a PL knot with a negligible effect on the conformal length.  By making a small
generic rotation, we can also assume that the edges of $K$ are transverse to the coordinate planes.

\vskip10pt

{\bf Step 1. Cutting blocks into pieces}

\vskip10pt

A key step in the proof is a lemma allowing us to cut a rectangular block into pieces that intersect $K$ nicely.

If a rectangular block has dimensions $L_1 \times L_2 \times L_3$ with $L_1 \le L_2 \le L_3$, then we
call the ratio $L_3/L_1$ the eccentricity of the block.

\newtheorem*{blockdec}{Block decomposition lemma}

\begin{blockdec} Let $Q \subset \mathbb{R}^3$ be a rectangular block of eccentricity $ < 10$, and suppose
that each face of $Q$ intersects $K$ at most $1000 convol_1(K)$ times.  Then we can decompose $Q$ into 
smaller blocks with eccentricity $< 10$ so that each face of each smaller block intersects $K$ at most
$1000 convol_1(K)$ times.

The number of smaller blocks will be at least 8 and at most 2000.  Each smaller block has maximal side
length at most $L_1(Q)/2$, where $L_1(Q)$ denotes the smallest side length of $Q$.

\end{blockdec}

\proof  Let $X$ denote the 2-skeleton of a
cubical lattice with side length $L_1(Q)/2$.  (We assume that the axes of $X$ are parallel to the axes of $Q$.)
We will make a random translation of $X$.  The translated $X$ decomposes $Q$ into rectangular
blocks.  We will prove by a probabilistic argument that some translations of $X$ decompose $Q$ into pieces
that satisfy our conclusions.

We begin with the easy conclusions.  Because $X$ is a lattice with spacing $L_1(Q)/2$, all of the blocks
in the decomposition have all side lengths $\le L_1(Q)/2$.  By the bound on the eccentricity of $Q$, 
the lengths $L_2(Q), L_3(Q)$ lie in the range $L_1(Q) \le L_2(Q), L_3(Q) \le 10 L_1(Q)$.  So the number of blocks in the decomposition is
at least 8 and at most $(3) (21) (21) \le 2000$.

Next we consider the probability that all the blocks in the decomposition have eccentricity $< 10$.
Since all blocks have maximal side length $\le L_1(Q) / 2$, it suffices to check that they have
minimal side length $> L_1(Q) / 20$.  The only thing that may go wrong is that one of the planes
of $X$ may lie within $L_1(Q) / 20$ of a parallel face of $Q$.  For each face, the probability of this
happening is $1/10$.  Therefore, with probability at least $(4/10)$, all of the cubes in the decomposition
have eccentricity $< 10$.

Next we consider the intersection $X \cap Q \cap K$.  We take the average over all translations of $X$.
Integral geometry will allow us to bound this average in terms of $length (K \cap Q)$.
Our $X$ is a union of planes in different directions.  We write $X_{12} \subset X$ to denote the union
of the planes in the $x_1-x_2$ direction, and so on.  Hence $X = X_{12} \cup X_{13} \cup X_{23}$.

Integral geometry tells us that the average number of intersections between $X_{12}$ and
$K \cap Q$ is

$$\le \frac{ length(K \cap Q) }{L_1(Q) / 2}. $$

Therefore, the average number of intersections between $X$ and $K \cap Q$ is 

$$\le \frac{ 3  length(K \cap Q) }{L_1(Q) / 2}. $$

Now $Q$ is contained in a ball of radius at most $10 L_1(Q)$.  Therefore,
the length of $K \cap Q$ is $\le 10 L_1(Q) convol_1(K)$.  So the average
number of intersections between $X$ and $K \cap Q$ is $\le 60 convol_1(K)$.

Therefore, we can find a translate of $X$ so that all the blocks in the decomposition
have eccentricity $< 10$ and also $X$ intersects $K \cap Q$ in $\le 150 convol_1(K)$
points.

Finally, we let $Q_i$ be any rectangular block in the decomposition.  Any face of $Q_i$
is either part of $X$ or part of a face of $Q$.  Hence the number of times that $K$
intersects this face is $\le 1000 convol_1(K)$. \endproof

{\bf Step 2. A tree of nested rectangles}

\vskip10pt

Using the block decomposition lemma repeatedly, we will cut space into a tree of nested rectangular blocks so that
each face of each block intersects $K$ in $\lesssim convol_1(K)$ points.  The largest rectangle will contain $K$, the smallest
rectangle will intersect $K$ in simple segments, and the regions between consecutive rectangles will intersect
$K$ in a controlled way.

We now give a more detailed description of the objects we will construct.  (We will construct them
below, by using the block decomposition lemma repeatedly.)
  
We let $Q_0$ be a large cube containing $K$.  The cube
$Q_0$ is sub-divided into rectangular blocks $B_i$.  Inside of each block $B_i$, we choose
a slightly smaller nested rectangle $Q_i$.  We will construct a mapping from $Q_0$ to a tree $T$.
The boundary of $Q_0$ and the boundaries of the $B_i$ are mapped to the root vertex $v_0$.
The boundaries of the $Q_i$ are mapped to vertices $v_i$ adjacent to $v_0$.  The region 
$B_i \setminus Q_i$ is mapped to the edge between $v_0$ and $v_i$.

Now we repeat the above construction inside each rectangle $Q_i$.  Each rectangle $Q_i$ is divided
into rectangular blocks $B_{ij}$.  Inside each $B_{ij}$, we choose a slightly smaller nested rectangle
$Q_{ij}$.  The boundaries of all the blocks $B_{ij}$ are mapped to $v_i$.  The boundary of $Q_{ij}$ is
mapped to $v_{ij}$, a vertex adjacent to $v_i$.  The region $B_{ij} \setminus Q_{ij}$ is mapped to the
edge from $v_i$ to $v_{ij}$.

This process repeats a large finite number of times.  At the end, we are left with a certain number of terminal
rectangles $Q_{\alpha}$, where $\alpha$ is a multi-index.  The boundary of $Q_{\alpha}$ is mapped to $v_{\alpha}$.
We let $x_{\alpha}$ denote a point in the interior of $Q_{\alpha}$ (we will choose $x_\alpha$ later).  
We map $x_\alpha$ to a vertex $v_{\alpha, +}$, and
we map the interior of $Q_\alpha$ to the edge from $v_{\alpha}$ to $v_{\alpha, +}$.

These objects obey the following estimates.

{\bf Property 1.} Each face of any $Q_{\alpha}$ or $B_{\alpha}$ intersects $K$ $\le 1000 convol_1(K)$ times.

{\bf Property 2.} Each region $B_{\alpha} \setminus Q_{\alpha}$ does not contain any of the vertices of $K$.
(Recall that $K$ is a PL curve consisting of finitely many straight edges connected at vertices.)
Moreover, each region $B_{\alpha} \setminus Q_{\alpha}$ meets $K$ in $\le 1000 convol_1(K)$
straight segments with length very small compared to the distance between the segments - and
each segment enters through an open face of $B_{\alpha}$ and exits through the corresponding
face of $Q_{\alpha}$.

{\bf Property 3.} Each terminal $Q_{\alpha}$ intersects $K$ in either part of a single edge, or one vertex
and part of two adjacent edges, or not at all.

{\bf Property 4.} Each $Q_{\alpha}$ is subdivided into at most 2000 blocks $B_{\alpha, j}$.

\vskip5pt

We now construct the $Q$'s and the $B$'s, and prove that they have Properties 1-4.

We let $Q_0$ be a large cube containing $K$ in its interior.  The boundary of $Q_0$ does not
intersect $K$, so $Q_0$ obeys all desired estimates.

We use the block decomposition lemma to divide $Q_0$ into rectangular blocks $B_i$.
The lemma tells us that $B_i$ obeys Property 1, and that the number of blocks $B_i$ is at most
2000, confirming Property 4.  By choosing the blocks $B_i$ in general position, we can assume that
the faces of $B_i$ never contain vertices of $K$ and that the 1-dimensional edges of
$\partial B_i$ never intersect $K$.  So the blocks $B_i$ obey all the desired Properties.  Moreover,
we know that each block $B_i$ has eccentricity $< 10$, and we know that the diameter of
$B_i$ is $\le (1/2) diam(Q_0)$.

Next we choose {\it slightly} smaller nested blocks $Q_i$ in each $B_i$.  By choosing $Q_i$
very close to $B_i$, we can arrange that $B_i \setminus Q_i$ does not contain any vertices of
$K$.  Then by choosing $Q_i$ even closer to $B_i$ we can arrange all of Property 2.  (At this
step, we use that the edges of $K$ are transverse to the coordinate planes and hence to the
faces of $B_i$.)  Since each edge of $K$ passing through a face of $Q_i$ passes through the
corresponding face of $B_i$, we see that $Q_i$ obeys Property 1 also.  Therefore, the cubes
$Q_i$ obey all the desired properties.  Moreover, since $Q_i$ is very close to $B_i$ we can 
arrange that the eccentricity of $Q_i$ is $< 10$, and we know that $diam(Q_i) < (1/2) diam (Q_0)$.

Now we proceed inductively.  Since $Q_i$ has eccentricity $< 10$ and each face of $Q_i$
meets $K$ in $< 1000 convol_1(K)$ edges, we can apply the block decomposition lemma to
divide $Q_i$ into blocks $B_{ij}$.  Just as above, we check that the $B_{ij}$ obey the Properties.
Then we choose a nested rectangular block $Q_{ij} \subset B_{ij}$ just slightly smaller than
$B_{ij}$, and we check as above that $Q_{ij}$ obeys all the desired properties.  Moreover,
the $Q_{ij}$ still have eccentricity $< 10$ and diameter $< (1/2)^2 diam (Q_0)$.  We can
then repeat this procedure as many steps as we want, and the diameters of the cubes decrease
exponentially as we go.

To decide when to stop, we let $\epsilon$ denote the minimum of the shortest length of an edge of $K$ and
the smallest distance between two non-adjacent edges of $K$.  We repeat the above subdivision
process $D$ times, where we choose $D$ so that $2^{-D} diam (Q_0) < \epsilon$.  Every
terminal cube $Q_\alpha$ has diameter $< \epsilon$, and so it obeys Property 3.

\vskip10pt

{\bf Step 3. A thin triangulation of $S^3$}

\vskip10pt

In Step 2, we cut $Q_0$ into pieces (nested rectangles) that intersect $K$ in a simple way.
Now we triangulate $Q_0$ by cutting each piece into simplices.  We do this so that
the knot $K$ ends up in the 1-skeleton of the triangulation, and at the same time, each of the pieces
is not too complicated.

We express our control of the triangulation in terms of a simplicial map from the triangulated
sphere to a tree $T$.  We say that a simplicial map from a triangulated 3-manifold to a graph is
non-degenerate if the image of each 3-simplex is an edge of the graph.  In other words, a non-degenerate
map does not collapse a 3-dimensional simplex to a vertex.  

\begin{lemma} If $K \subset \mathbb{S}^3$ is a knot, then we can triangulate $S^3$ by a triangulation
$Tri_0$ containing $K$ in its 1-skeleton with the following property.  There is a non-degenerate simplicial map
$\pi_0$ from $(S^3, Tri_0)$ to a tree $T$ obeying the following estimates.

1. For each edge $e$ of $T$, the inverse image $\pi_0^{-1}(e)$ contains $\lesssim convol_1(K)$ simplices.

2. Each vertex of $T$ lies in $\lesssim 1$ edges of $T$.

\end{lemma}

\begin{proof} Our triangulation is going to be built around the nested blocks from Step 2.  In particular,
the boundary of each $B_\alpha$ will be contained in the 2-skeleton of the triangulation.  First
we triangulate the faces of each $\partial B_\alpha$.  Second we triangulate the regions $B_\alpha \setminus
Q_\alpha$.  Finally, we triangulate the terminal cubes $Q_\alpha$ and the exterior of $Q_0$.

First we triangulate the faces of $B_i$ - the blocks at the first step.  Each face meets $K$ in
$\le 1000 convol_1(K)$ points.  We use the Delaunay triangulation of each face using these
points.  So each face has $\le 10^4 convol_1(K)$ simplices.

Next we triangulate the faces of $Q_i$.  Because of Property 2 in Step 2, we can move the
triangulation of each face of $B_i$ to an equivalent triangulation of each face of $Q_i$.  In fact,
we can lift the triangulation of $\partial B_i$ to a polyhedral structure on $B_i \setminus Q_i$
where the 3-dimensional faces are convex polyhedra combinatorially equivalent to $\Delta^2 \times [0, 1]$.
Each segment of $K \cap (B_i \setminus Q_i)$ is contained in the 1-skeleton of this polyhedral
structure.

Next we triangulate the faces of $B_{ij}$.  Some faces of $B_{ij}$ lie in $\partial Q_i$.  We call
these boundary faces.  Other faces of $B_{ij}$ lie in the interior of $Q_i$, and we call them interior
faces.  Each face of $Q_i$ is subdivided into four boundary faces of four different blocks $B_{ij}$.
We have already triangulated each face of $Q_i$, but we need to refine the triangulations to 
include the edges separating the four boundary faces.  After we add these edges, each face of
$B_{ij}$ is divided into polygons.  Each polygon is part of one of the simplices
in the face of $Q_i$, and so the number of polygons is $\le 10^4 convol_1(K)$.  Also, each polygon
is the intersection of a simplex with a rectangle, and so it has $\le 7 $ sides. 

Each interior face of $B_{ij}$ we triangulate so that the intersections with $K$ are vertices.  We
again use the Delaunay triangulation, and so we get a triangulation with $\le 10^4 convol_1(K)$
simplices.

Now we continue inductively to finer and finer scales.  At each scale, we get a polyhedral structure 
on each face of $B_{\alpha}$ with $\le 10^4 convol_1(K)$ polygons which are each the intersection of a simplex
with a rectangle.  Hence they each have $\le 7$ sides.  We also get a polyhedral structure on each
region $B_{\alpha} \setminus Q_{\alpha}$ with $\le 10^5 convol_1(K)$ polyhedra, where each 3-face 
has $\lesssim 1$ 2-faces in its boundary.  We subdivide these polyhedra without adding vertices to get a triangulation.
In each region $B_\alpha \setminus Q_{\alpha}$, this triangulation has $\lesssim convol_1(K)$ faces.

Finally, we do the terminal cubes $Q_\alpha$.  We have already triangulated the boundary of $Q_\alpha$.
If $K \cap Q_\alpha$ is a segment, then we let $x_\alpha$ be a point on the segment in the interior of $Q_\alpha$.
If $K \cap Q_\alpha$ is a vertex with two segments adjacent to it, we let $x_\alpha$ be the vertex, which lies in
the interior of $Q_\alpha$.  If $K \cap Q_\alpha$ is empty, then we just let $x_\alpha$ be any point in the
interior of $Q_\alpha$.  To triangulate the interior of $Q_\alpha$, we take the cone of the triangulation of $\partial Q_\alpha$ with vertex $x_\alpha \in Q_\alpha$.  Because of our choice of $x_\alpha$, the intersection $K \cap Q_\alpha$ lies
in the 1-skeleton of the triangulation.  It has $\lesssim 1$ simplices.

Since $\partial Q_0$ did not intersect $K$, we triangulate it with $\lesssim 1$ faces, and then we can triangulate
$S^3 \setminus Q_0$ by taking a cone from the triangulation of $\partial Q_0$ to the point at infinity.  This completes
our triangulation of $S^3$.

The tree $T$ is basically the same as in Step 2, but to deal with $S^3 \setminus Q_0$ we add one more vertex
$v_\infty$ and one more edge from $v_\infty$ to $v_0$.  We map the point at $\infty$ to $v_\infty$.  We map
all the vertices in the boundary of $B_i$ to $v_0$.  We map all the vertices in the boundary of $B_{ij}$ to $v_i$,
and so on.  For a terminal cube $Q_\alpha$, we map all the vertices in $\partial Q_\alpha$ to $v_{\alpha}$.
We map the vertex $x_{\alpha}$ in the center of a terminal cube $Q_\alpha$ to $v_{\alpha, +}$.  Then we extend
simplicially.  The inverse image of an edge from $v_{\alpha}$ to $v_{\alpha,j}$ is the region $B_{\alpha,j} \setminus
Q_{\alpha,j}$, which has $\lesssim convol_1(K)$ simplices.  The inverse image of a terminal edge
from $v_{\alpha}$ to $v_{\alpha, +}$ is the terminal cube $Q_\alpha$, which has $\lesssim 1$ simplices.
Similarly, the inverse image of the edge from $v_0$ to $v_\infty$ is $S^3 \setminus Q_0$, which has
$\lesssim 1$ simplices.

As in Step 2, the tree $T$ has degree $\le 2001 \lesssim 1$.

\end{proof}

{\bf Step 4. A thin triangulation of $M^3$}

\vskip10pt

Using our cover $F: M^3 \rightarrow S^3$ which is ramified over the knot $K$, we can pull the
triangulation from Step 3 up to $M$.  The resulting triangulation of $M$ retains all the good
characteristics of the triangulation of $S^3$.

\begin{lemma} Suppose that $M^3$ is a closed oriented 3-manifold.  Suppose that $F: M^3
\rightarrow S^3$ is a degree 3 cover ramified over a knot $K \subset S^3$.  Then, there is a
triangulation $Tri$ of $M^3$ and a non-degenerate simplicial map
$\pi$ from $(M^3, Tri)$ to a tree $T$ obeying the following estimates.

1. For each edge $e$ of $T$, the inverse image $\pi^{-1}(e)$ contains $\lesssim convol_1(K)$ simplices.

2. Each vertex of $T$ lies in $\lesssim 1$ edges of $T$.

\end{lemma}

\begin{proof} Let $Tri_0$ be the triangulation of $S^3$ constructed in Step 3.  Since
$K$ is contained in the 1-skeleton of $Tri_0$, we can lift $Tri_0$ to a triangulation $Tri$
of $M^3$.  The map $F: (M^3, Tri) \rightarrow (S^3, Tri_0)$ is a simplicial map.
If $\Delta$ denotes a simplex of $Tri_0$ which is not contained in the knot
$K$, then $\Delta$ lifts to three simplices in $Tri$.  If $\Delta$ is an edge or
vertex of $Tri_0$ which is contained in $K$, then $\Delta$ lifts to a single edge or
vertex of $Tri$. 

Let $\pi_0: (S^3, Tri_0) \rightarrow T$ be the simplicial map to a tree constructed in Step 3.
We define $\pi: (M^3, Tri) \rightarrow T$ to be the compositon $\pi_0 \circ F$.  This is also
a simplicial map.  The map $F$ is non-degenerate in the sense that 
each 3-simples of $M^3$ is mapped to a whole 3-simplex of $S^3$.  Therefore, $\pi$ is
non-degenerate as well.  The number of simplices in $\pi^{-1}(e)$ is at most three times
the number of simplices of $Tri_0$ in $\pi_0^{-1}(e)$, and so $\pi$ obeys Property 1.
The tree $T$ is the same tree as in Step 3, and so it obeys Property 2. \end{proof}

{\bf Step 5.  Hyperbolic geometry}

\vskip10pt

At this point, we connect the hyperbolic geometry of $M$ with the combinatorial information
about the triangulation $(M, Tri)$.  In a bit more generality, we prove the following lemma.

\begin{lemma} Suppose that $(M^3, Tri)$ is a triangulated manifold (or pseudomanifold),
$\pi: (M^3, Tri) \rightarrow T$ is a non-degenerate simplicial map to a tree $T$ of degree
$\lesssim 1$, and that
for each edge $e \subset T$, $\pi^{-1}(e)$ has $\le A$ simplices.  On the other hand,
suppose that $M^3$ admits a map of non-zero degree modulo 2 to a closed hyperbolic
manifold $(N^3, hyp)$ with volume $V$ and Cheeger constant $h$.

Then $A \gtrsim h V$.

\end{lemma}

Given Lemma 4.3 and Lemma 4.4, we can quickly prove Theorem 4.2.
We let $Tri$ be the triangulation of $M^3$ from Lemma 4.3.  
Now $(M^3, Tri)$ satisfies the first hypothesis of Lemma 4.4.
$A \lesssim convol_1(K)$.  Then we take $(N^3, hyp)$ to be $(M^3, hyp)$.  The identity
map from $(M^3, Tri)$ to $(M^3, hyp)$ has degree 1 modulo 2.  All of the hypotheses
of Lemma 4.4 are satisfied, and we conclude that  $convol_1(K) \gtrsim h V$.

Now we turn to the proof of Lemma 4.4.

\begin{proof}[Proof of Lemma 4.4]  We first apply simplex straightening to the map $\phi$.  After straightening, we may
assume that $\phi(\Delta^2)$ has area $\lesssim 1$ for each 2-simplex $\Delta^2$ in $(M^3, Tri)$
and that $\phi(\Delta^3)$ has volume $\lesssim 1$ for each 3-simplex $\Delta^3$ in $(M^3, Tri)$.

Now we can describe the intuition behind the lemma.  
The first hypothesis roughly means that the triangulated manifold $M$ is ``thin".  It looks morally
like a thin neighborhood of a tree with thickness $A$.  On the other hand, $h V$ is the ``thickness" of $(N^3, hyp)$.
Since $M$ fits around $(N^3, hyp)$, it forces $M$ to be as thick as $(N^3, hyp)$ and so
$A \gtrsim h V$.  The detailed proof follows.

The map $\pi: (M^3, Tri) \rightarrow T$ allows us to break $(M^3, Tri)$ into smaller pieces.
Namely, for each edge $e \subset T$, we define $Y(e) = \pi^{-1}(e)$.  We view $Y(e)$ as
a mod 2 3-chain.  Because $\pi$ maps every 3-simplex to exactly one edge of $T$,
$\sum_{e \subset T} Y(e)$ is homologous to the fundamental class $[M]$.  

Next the boundary of $Y(e)$ has a nice structure.  Suppose that $\partial e$ consists of the
vertices $v_1$ and $v_2$.  Then $\partial Y(e)$ breaks into two parts, one over $v_1$ and
one over $v_2$.  We write $\partial Y(e) = Z(e, v_1) + Z(e, v_2)$, where $Z(e,v) \subset 
\pi^{-1}(v)$.  For any pair $v \in e$, where $v$ is a vertex of $T$ and $e$ is an edge of $T$
containing $v$, we have defined a 2-cycle $Z(e,v)$ in $M^3$ contained in $\pi^{-1}(v)$.  Since
$Z(e,v)$ is part of the boundary of $Y(e)$, each $Z(e,v)$ contains $\lesssim A$ simplices.

Because $(M^3, Tri)$ is itself a 3-cycle, for each vertex $v$, $\sum_{e: v \in e} Z(e,v) = 0$.
This holds because $\partial (\sum_{e \subset T} Y(e) ) = 0$, and the part of this boundary
in $\pi^{-1}(v)$ is exactly $\sum_{e: v \in e} Z(e,v)$.

Also, because $T$ is a tree, each $Z(e,v)$ is null-homologous in $M^3$.  To see this, let $T_0$ be the
component of $T \setminus v$ which contains $e$.  Then $\pi^{-1}(T_0)$ defines a 3-chain with
boundary $Z(e,v)$.

Next we consider mapping the $Y$'s and $Z$'s to $(N^3, hyp)$.  We define
$\bar Y(e)$ to be $\phi( Y(e) )$ and $\bar Z (e, v)$ to be $\phi (Z(e,v))$.  By the bound on
the number of simplices of $Y$ and $Z$ and the straightness of $\phi$, we can conclude
that $\bar Y(e)$ has volume $\lesssim A$ and $\bar Z(e,v)$ has area $\lesssim A$.

Since $Z(e,v)$ was null-homologous in $M^3$, each $\bar Z(e,v)$ is null-homologous in $(N^3,
hyp)$.  Therefore, each $\bar Z(e,v)$ bounds a chain $\bar Y(e,v)$ with volume $\lesssim h^{-1} A$.

Now the sum $\sum_e \bar Y(e)$ is homologous to $(deg \phi) [N]$ in $H_3(N, \mathbb{Z}_2)$.
Since we assumed the degree of $\phi$ is non-zero in $\mathbb{Z}_2$, the cycle $\sum_e \bar Y(e)$
is non-trivial.  Using the $Z$'s and $Y$'s, we will break this cycle into small pieces.

$$ \sum_e \bar Y(e) = \sum_e \bar Y(e) + \sum_{v,e : v \in e} \bar Y(e,v) + \sum_{v,e: v \in e} \bar Y(e,v) = $$

$$ \sum_e ( \bar Y(e) + \sum_{v \in e} \bar Y(e,v) ) + \sum_v ( \sum_{e: v \in e} \bar Y(e,v) ) . $$

Each term in parentheses on the last line is a cycle.  The reader may check this by computing the boundaries
and recalling that $\sum_{e: v \in e} Z(e,v) = 0$.  One of these cycles must be topologically non-trivial
and so have volume $\ge V$.  But each term in parentheses has volume $\lesssim (1 + h^{-1}) A
\lesssim h^{-1} A$.  Therefore, $h^{-1} A \gtrsim V$ as we wanted to show.  \end{proof}

This completes the proof of Theorem 4.2. \end{proof}

There are many open questions about distortion and conformal length/volume.  It would be interesting to know
more about the distortion of many particular knots, starting with the torus knot $T_{2,q}$.  One may also ask
about the distortion and conformal length of graphs embedded in $\mathbb{R}^3$.  Then one may ask about
the distortion and conformal volume of simplicial complexes $X^k$ embedded in $\mathbb{R}^n$.  One may
also ask about the distortion and conformal volume of higher dimensional knots. 

One may also ask about the combinatorial thickness or retraction thickness of knots.  There are various
results involving different kinds of `thickness' of knots, including \cite{Na} and \cite{BuSi}.

\section{Appendix: geometry and topology of arithmetic hyperbolic 3-manifolds}

In this section, we discuss closed arithmetic hyperbolic 3-manifolds, which area a special class of closed hyperbolic 3-manifolds.  Arithmetic
means that the fundamental group is an arithmetic lattice in the group of isometries of
hyperbolic space.  This arithmetic property turns out to have important consequences
in geometry and topology.

To fix ideas, we consider a single closed oriented arithmetic hyperbolic 3-manifold
$X_0$ and then consider a sequence of arithmetic covers $X_i \rightarrow X_0$
with degree $D_i \rightarrow \infty$.  The volume of $X_i$ is $V_i = D_i Vol(X_0)$.

The arithmetic property of the sequence of covers leads to a remarkable geometric
inequality about the $X_i$.  Namely, the Cheeger constant $h(X_i)$ is bounded below
uniformly: $h(X_i) \ge c > 0$ for all $X_i$.  (The first estimate of this kind was proven
by Selberg \cite{S}, published in 1965.  Selberg proved that the first eigenvalue of the
Laplacian $\lambda_1(X_i) \ge c > 0$ where $X_i$ are arithmetic surfaces.  A lower bound
on $\lambda_1$ and a lower bound on the Cheeger constant $h$ are closely related by
the theorems of Cheeger \cite{C} and Buser \cite{Bu}.  In particular, Buser's theorem implies
that if a sequence of hyperbolic manifolds has $\lambda_1(X_i) \ge c > 0$, then it also has
$h(X_i) \ge c' > 0$.  The most general theorem about arithmetic manifolds and expanders
is due to Clozel \cite{Cl}.) 

If we fix a triangulation of $X_0$, then we can lift it to
give a family of triangulations of $X_i$.  We let $\Gamma_i$ denote the 1-skeleton of the
triangulation of $X_i$.  The isoperimetric constant $h(\Gamma_i)$ approximately agrees
with $h(X_i)$.  In particular, $h(\Gamma_i) \ge c' > 0$ for all $i$.  Hence the 
graphs $\Gamma_i$ are a family of expanders.

Now we turn to the topological implications of geometric information about $X_i$.  The first
example that I know about is a theorem of Milnor and Thurston \cite{MT}, which proves that
a closed hyperbolic manifold with large volume is topologically complicated.  Milnor and Thurston
prove that if $X$ is a closed hyperbolic manifold $X$ with volume $V$, then it takes $\gtrsim
V$ simplices to triangulate $X$.  (Moreover, it takes $\gtrsim V$ simplices to build a singular cycle
homologous to $[X]$.)  Their work introduced the important idea of simplex straightening.

Notice that our arithmetic $X_i$ can be triangulated by $C_0 D_i$ simplices, where $C_0$
is the number of simplices needed to triangulate $X_0$.  Therefore, the Milnor-Thurston bound
is nearly sharp: $X_i$ can be triangulated with $\lesssim V_i$ simplices, and any triangulation
requires $\gtrsim V_i$ simplices.

The work of Milnor and Thurston connected $Vol(X_i)$ to the topology of $X_i$.  More recent
work has connected the isoperimetric bound $h(X_i) \ge c > 0$ with the topology of $X_i$.
A typical application is that for arithmetic hyperbolic 3-manifolds, the Heegaard genus of $X_i$ is $\gtrsim V_i$.

Here is a loose sketch of the proof, emphasizing the main idea.  Let $G_i$ be the Heegaard genus
of $X_i$.  In particular, we can find a smooth map $F: X_i \rightarrow \mathbb{R}$ so that each fiber
$F^{-1}(y)$ is a closed surface of genus $\le G_i$.  Next we ``simultaneously straighten" all of the fibers.
This step requires care and more detail, but it turns out to be morally correct.  Now each fiber has
area $\lesssim G_i$.  One of the fibers must bisect $X_i$, in the sense that half of the volume of $X_i$
lies on each side.  Now the isoperimetric inequality says that $G_i \gtrsim h_i (1/2) V_i \gtrsim V_i$.

This estimate for Heegaard genus is also sharp up to a constant factor.  The Heegaard genus of $X_i$ is $\le D_i 
G_0$, as one sees by lifting a Heegaard decomposition of $X_0$.  Hence the Heegaard genus of
$X_i \sim V_i$.

An early paper with many of these ideas
is \cite{BCW}, which estimates the Heegaard genus of a closed hyperbolic manifold containing
a large embedded ball.  Bachmann, Cooper, and White prove that a closed hyperbolic 3-manifold
containing an embedded ball of radius $R$ has Heegaard genus at least $(1/2) \cosh(R) \ge (1/4) e^R$.
This number is essentially the area of a cross-section of the ball.  
Arithmetic examples do have embedded balls of radius $R_i \rightarrow \infty$, 
and so one sees that the Heegaard genus of $X_i$ is at least $V_i^{\epsilon}$ for some $\epsilon > 0$.

The sketch above suggests a little bit more than a lower bound on the Heegaard genus of $X_i$.
Namely, any smooth map $F: X_i \rightarrow \mathbb{R}$ needs to have a complicated fiber.

\newtheorem*{comfib}{Complicated fiber inequality}

\begin{comfib} (\cite{Gr1}) Suppose that $X_i$ are arithmetic hyperbolic 3-manifolds and that $F: X_i \rightarrow \mathbb{R}$ is a generic smooth map.  Then one
of the fibers of $F$ has the sum of its Betti numbers $\gtrsim V_i$.
\end{comfib}

This inequality should be compared with the following basic inequality about expanders: if $\Gamma_i$ is a family 
of expanders and $\Gamma_i$ has $N_i$ edges, then any map $F: \Gamma_i \rightarrow \mathbb{R}$ has a fiber that meets
$\gtrsim N_i$ different edges.

Starting from here, arithmetic hyperbolic 3-manifolds can be considered as topological analogues of expanders.
This point of view leads to the theorems in \cite{Gr1} and in this paper.

It is not known whether this complicated fiber inequality (or something similar) applies to arithmetic hyperbolic manifolds of dimension $k \ge 4$.

\end{document}